\documentclass[10pt,reqno]{amsart}
\usepackage[foot]{amsaddr}

\usepackage[hmarginratio=1:1,top=32mm,columnsep=20pt]{geometry} 
\usepackage[utf8]{inputenc}
\usepackage{amsmath}
\usepackage{graphicx}
\usepackage{amssymb}
\usepackage{esint}
\usepackage{color}
\usepackage{amsthm}
\usepackage{epsfig}
\usepackage[]{todonotes}
\usepackage{mathrsfs} 
\usepackage{xxcolor}%

\usepackage{algorithm}
\usepackage{algpseudocode}
\usepackage[pagebackref=false,linktocpage=true,colorlinks=true,linkcolor=Blue,citecolor=Green]{hyperref}
\usepackage{doi}
\usepackage{enumitem}
\usepackage{mathtools}
\usepackage{tikz}
\usetikzlibrary{arrows}
\usepackage[english]{babel}
\usepackage[square,numbers]{natbib}
\iffalse 
\usepackage[autostyle]{csquotes}
\usepackage[
    backend=biber,
    style=numeric,
    firstinits = true,
    isnb=false,
    issn=false,
    url=false, 
    doi=true,
    eprint=true
]{biblatex}

\renewbibmacro{in:}{
  \ifentrytype{article}{}{
  \printtext{\bibstring{in}\intitlepunct}}}

\DeclareFieldFormat[article,unpublished,book]{pages}{#1}  

\DeclareFieldFormat[article]{title}{#1}
\DeclareFieldFormat[unpublished]{title}{#1}
\DeclareFieldFormat[misc]{title}{#1}
\fi 
\newtheorem{theorem}{Theorem}
\newtheorem{proposition}[theorem]{Proposition}
\newtheorem{lemma}[theorem]{Lemma}

\newtheorem{definition}[theorem]{Definition}
\newtheorem*{theorem*}{Theorem}
\newtheorem*{maintheorem*}{Main Result}
\newtheorem*{theorem**}{{``Theorem''}}

\theoremstyle{definition}

\newtheorem{ass}[theorem]{Assumption}
\newtheorem{remark}[theorem]{Remark}

\reversemarginpar

\newtheorem{innercustomgeneric}{\customgenericname}
\providecommand{\customgenericname}{}
\newcommand{\newcustomtheorem}[2]{%
  \newenvironment{#1}[1]
  {%
   \renewcommand\customgenericname{#2}%
   \renewcommand\theinnercustomgeneric{##1}%
   \innercustomgeneric
  }
  {\endinnercustomgeneric}
}

\newcustomtheorem{customthm}{Assumption}
\newcustomtheorem{customthm2}{Definition}

\def\XXint#1#2#3{{\setbox0=\hbox{$#1{#2#3}{\int}$ }
\vcenter{\hbox{$#2#3$ }}\kern-.6\wd0}}

\definecolor{Yellow}{rgb}{0.95,0.9,0.0} 
\definecolor{Red}{rgb}{0.8,0.1,0.1}
\definecolor{Green}{rgb}{0.1,0.65,0.2}
\definecolor{Blue}{rgb}{0.1,0.1,0.8}
\definecolor{Purple}{rgb}{0.7,0.1,0.7}
\definecolor{Grey}{rgb}{0.6,0.6,0.6}

\newcommand{\ov}{\overline}

\newcommand{\domain}{{\mathbb{T}^d}}

\newcommand{\m}{\,\mathrm{d}}				%use case: write \m\f{x}	for dx at the end of integral

\newcommand{\Ghd}{G_{h,d}} 				%our discrete grid on Td

\newcommand{\Ih}{\mathcal{I}_{h}}		%Interpolation operator
\newcommand{\rhor}[1][ ]{{\rho_{#1}}}			%Solution to DK
\newcommand{\rhobr}[1][ ]{{\ov{\rho}_{#1}}}	%Intermediate MFL
\newcommand{\empmeas}[1][ ]{{\mu_{{#1}}^{N}}}

\newcommand{\Ts}{{T_\oslash}}

\newcommand{\lev}{\ell}

\newcommand{\IndSet}[1]{I_{#1}}
\newcommand{\Four}[2]{e^{i{#2}\cdot{#1}}}
\newcommand{\F}[1]{e^{i#1(\cdot)}}
\newcommand{\Grid}[2]{G_{#1,#2}}
\newcommand{\rhob}[1][]{\overline{\rho}_{#1}}
\newcommand{\EstML}{\mu_{MLMC}}

\newcommand{\btt}[1]{t_{\leftarrow}^{#1}}

\newcommand{\tfactor}{\kappa_t}

\newcommand{\hl}{h_{\lev}}

\newcommand{\Dt}{\tau}
\newcommand{\Dtl}{\tau_{\lev}}
\newcommand{\cfl}{\mu}
\newcommand{\Tgrid}[1][]{\mathcal{S}_{#1}}
\newcommand{\ths}{2}

\allowdisplaybreaks[2] 		%1-4 for finer control of page breaks in amsmath environments
\numberwithin{equation}{section}	%only count equations etc within a section
\numberwithin{theorem}{section}

\allowdisplaybreaks[2]

\usepackage{delimdelim}

\begin{document}

% insert Title/short title
\title[MLMC methods for models of Fluctuating Hydrodynamics]{Multilevel Monte Carlo methods for the Dean--Kawasaki equation from Fluctuating Hydrodynamics}

\author{Federico Cornalba$^*$}
\address{$^*$University of Bath, Claverton Down, BA2 7AY, Bath, United Kingdom}
\email{\href{mailto:fc402@bath.ac.uk}{fc402@bath.ac.uk}}
%\thanks{$^*$University of Bath, Claverton Down, BA2 7AY, Bath, United Kingdom. %\emph{E-mail address:} \href{mailto:fc402@bath.ac.uk}{fc402@bath.ac.uk}
%}
\author{Julian Fischer$^\dag$}
\address{$^{\dag}$Institute of Science and Technology Austria (ISTA), Am~Campus~1, 
3400 Klosterneuburg, Austria. 
%\\ \emph{E-mail address:} \href{mailto:julian.fischer@ista.ac.at}{julian.fischer@ista.ac.at}
}
\email{\href{mailto:julian.fischer@ista.ac.at}{julian.fischer@ista.ac.at}}

% insert abstract
\begin{abstract}
Stochastic PDEs of Fluctuating Hydrodynamics are a powerful tool for the description of fluctuations in many-particle systems. 
In this paper, we develop and analyze a Multilevel Monte Carlo (MLMC) scheme for the Dean--Kawasaki equation, a pivotal representative of this class of SPDEs. 
We prove analytically and demonstrate numerically that our MLMC scheme provides a significant reduction in computational cost (with respect to a standard Monte Carlo method) in the simulation of the Dean--Kawasaki equation.
Specifically, we link this reduction in cost to having a sufficiently large average particle density, and show that sizeable cost reductions can be obtained even when we have solutions with regions of low density. 
Numerical simulations are provided in the two-dimensional case, confirming our theoretical predictions.

Our results are formulated entirely in terms of the law of distributions rather than in terms of strong spatial norms: this crucially allows for MLMC speed-ups altogether despite the Dean--Kawasaki equation being highly singular.

\end{abstract}
\thanks{{\bfseries Funding}: Both authors gratefully acknowledge funding from the Austrian Science Fund (FWF) through the project F65}
\maketitle


{\small{\bfseries MSc codes}. 65C05, 60H15, 35R60, 65N06, 82M36, 82C22.
}

\section{Introduction}
In the regime of large particle numbers, the behavior of interacting particle systems often gives rise to a classical PDE description, for instance in form of the classical equations of continuum mechanics. For medium-sized particle systems consisting of only, e.g., $10^5$--$10^9$ particles, this idealized description often becomes insufficient, as thermal fluctuations may begin to impact the behavior.
The theory of Fluctuating Hydrodynamics augments the classical PDEs of continuum mechanics with suitable noise terms to account for thermal fluctuations, giving rise to SPDE models; we refer the reader to Landau and Lifshitz \cite{LandauLifshitzVol6},  Spohn \cite{SpohnBook}, and te Vrugt, L\"owen and Wittkowski \cite{te2020classical}.

In this work, we are concerned with the numerical approximation of one of the most basic equations of Fluctuating Hydrodynamics, the \emph{Dean--Kawasaki equation}
\begin{align}
\partial_t \rho = \frac{1}{2}\Delta \rho + \nabla \cdot (\rho (\nabla V \ast \rho)) + N^{-1/2} \nabla \cdot \left(\sqrt{\rho}\xi\right).
\label{DK}
\end{align}
It describes the effective behavior of the density $\rho$ of a system of $N$ weakly interacting diffusing particles $\{X_i(t)\}_{i=1}^{N}$ in the regime of large particle numbers $N\gg 1$. Here, $V$ is a (sufficiently regular) interaction potential and $\xi$ denotes vector-valued space-time white noise. While the McKean-Vlasov equation $\partial_t \bar\rho =\frac{1}{2}\Delta \bar\rho + \nabla \cdot (\bar\rho (\nabla W \ast \bar\rho))$, obtained formally in the limit $N\rightarrow \infty$, describes the mean-field limit profile $\bar\rho$ of the particle system, the Dean--Kawasaki equation \eqref{DK} for $N\gg 1$ in addition captures the law of the density fluctuations due to the finite number of particles.

The Dean--Kawasaki equation \eqref{DK} is a highly singular SPDE; namely, it is even too singular to be renormalized by approaches like regularity structures or paracontrolled calculus \cite{HairerRegularityStructures,GubinelliImkellerPerkowskiParacontrolled}. As shown in \cite{KonarovskyiVonRenesse,KonarovskyiRenesse2}, the only martingale solutions to \eqref{DK} are in fact given by the empirical measures of the underlying particle system, that is 
$
\rho(x,t) \equiv \empmeas[t](x)  := N^{-1}\sum_{i=1}^{N}{\delta(x-X_i(t))}.
$
Nevertheless, a justification of SPDEs related to \eqref{DK} has recently been given in \cite{FehrmanGess,DirrFehrmanGess} in terms of large-deviation principles. Justifications for Dean--Kawasaki type models with regularized noise have been developed in \cite{DirrFehrmanGess,djurdjevac2024weak}.

Recently, the authors and collaborators \cite{cornalba2021dean,CornalbaFischerIngmannsRaithel} have shown that the Dean--Kawasaki equation may be viewed as a recipe for accurate and efficient simulations of the density fluctuations in the interacting particle system: When applying a formal spatial semi-discretization to \eqref{DK}, the law of density fluctuations predicted by the discretization accurately describes the law of density fluctuations in the underlying particle system, as long as the grid size $h$ is such that $Nh^d \gg 1$ ($d$ is the spatial dimension) -- that is, as long as on average there is substantially more than one particle per grid cell.

Nevertheless, computing statistical properties of particle density fluctuations $\bar\rho$ such as variances $\mean{| N^{1/2}\int (\empmeas[T]-\overline{\rho}^T)(x) \varphi(x)\m x|^2}$ (or, more generally, of random variables of the form 
\begin{align}\label{particle_fluct}
Q = \normalsize\psi\left(N^{1/2}\int (\mu^T_N - \overline{\rho}^T)(x)\varphi(x)\m x  \right)
\end{align}
for sufficiently regular test functions $\psi,\varphi$)
via discretizations of \eqref{DK} is still a computationally demanding task, as in addition to solving an SPDE one needs to sample over many realizations of the noise. Multilevel Monte Carlo (MLMC) methods are a powerful numerical tool that often allows to offset a large amount of the sampling cost by performing the computation for most samples only on a much coarser numerical grid. Some of the first applications of MLMC methods have been in the context of statistical physics \cite{brandt1994optimal} and stochastic differential equations in mathematical finance \cite{Giles}; they have since found widespread applications for random and stochastic PDEs, see for instance \cite{BarthLangSchwab,BarthSchwabZollinger,
CharrierScheichlTeckentrup,CliffeGilesScheichlTeckentrup,doi:10.1137/18M1175239,GilesReview,GilesReisinger,khodadadian2019multilevel}.

In this work, we develop and analyze MLMC methods to approximate $\mean{Q}$, for $Q$ of type \eqref{particle_fluct}, using discretizations of the Dean--Kawasaki equation \eqref{DK}. 
To the best of our knowledge, our present work is one of the first mathematical results on Multilevel Monte Carlo methods both for singular stochastic PDEs, and for stochastic PDEs involving multiplicative space-time white noise.

\subsection{Summary of main results}

In order to approximate $\mean{Q}$ (see \eqref{particle_fluct}), we set up a sequence of random variables $(P_{\ell})_{\ell}$ associated with space-time discretizations of \eqref{DK} over levels $\ell=0,\dots,L$ (the higher the $\lev$, the smaller the spatial grid size $h_\ell$ and time step\footnote{since we will link $\hl$ and $\tau_{\ell}$ via a standard Courant-Friedrichs-Lewy (CFL for short) condition, many of our considerations will be stated -- for notational convenience -- just in terms of the spatial discretization parameters $(\hl)_{\ell}$.} $\tau_\ell$, and the more accurately $P_\ell$ approximates $Q$). A key tool of our analysis is to couple the noises of discretized versions of \eqref{DK} on consecutive levels $\ell-1,\ell$. We provide two different ways of doing this. The first way is to couple the noises over the Fourier frequencies, thus having the noise on level $\ell-1$ share all of its frequencies with the noise on level $\ell$ (``Fourier coupling", cfr \& \ref{secFourierCoupling}).
The second way is to construct the noise on level $\ell-1$ using only local spatial information coming from the noise on level $\ell$ (``Right-Most Nearest Neighbours (NN) coupling", cfr \&\ref{secRM}).

We can now informally state our main result:
\begin{theorem}\label{main_res}
Let $\varepsilon > 0$ be a given accuracy. Consider a sequence of levels for which the finest grid size $h_{\min}$ satisfies $h^2_{\min} \lesssim \varepsilon$. Under the technical assumptions \ref{N_h_scaling}--\ref{reg_mfl}--\ref{ass_stability}--\ref{reg_test_func}--\ref{init_datum}, which in particular require the average particle density $Nh^d_{\min}$ to satisfy  
\begin{align}\label{informal_lower_bound_density}
Nh^d_{\min} \gtrsim h_{\min}^{-\beta},\qquad
\mbox{where }\beta :=  
\left\{
\begin{array}{ll}
4,  & \mbox{for Fourier coupling, cfr \&\ref{secFourierCoupling}} \\
2, & \mbox{for NN coupling, cfr \&\ref{secRM}} \\
\end{array}
\right.
\end{align}
we can set up a Multilevel Monte Carlo estimator $\mu_{MLMC}$ for the simulation of \eqref{DK} which achieves a mean-square error $\mean{\left|\mu_{MLMC} - \mean{Q}\right|^2}$ of order $\varepsilon^2$, while carrying a computational cost $C_{MLMC}$ bounded by
\begin{align}\label{MLMC_complexity}
\mean{C_{MLMC}} \lesssim 
\left\{
\begin{array}{ll}
\varepsilon^{-2}  & \mbox{for } d = 1\mbox{ and Fourier coupling, cfr \&\ref{secFourierCoupling}} \\
\varepsilon^{-2}\cdot(\log\varepsilon)^2, & \mbox{for } d = 2\mbox{ and Fourier coupling, cfr \&\ref{secFourierCoupling}} \\
\varepsilon^{-2}\cdot\varepsilon^{-(d-2)/2}, & \mbox{for } d \geq 3\mbox{ and Fourier coupling, cfr \&\ref{secFourierCoupling}}\\
\varepsilon^{-2}\cdot\varepsilon^{-d/2}, & \mbox{for } d \geq 1\mbox{ and NN coupling, cfr \&\ref{secRM}} \\
\end{array}
\right.
\end{align}
as opposed to the standard MC method, whose cost bound is $$\mean{C_{MC}} \lesssim \varepsilon^{-2}\cdot \varepsilon^{-d/2-1}.$$ 
%In particular, this Multilevel Monte Carlo method achieves a variance reduction of factor 
%\begin{align*}
%F \propto h_{\min}^{-[(d + 2)\wedge 4]}\cdot |\log(h_{\min})|^{-1}
%\end{align*}
%\todo{add $L^\star$} with respect to the associated basic Monte Carlo method with the same grid size $h_{\min}$.
\end{theorem}

As a key challenge for our analysis, unlike in Multilevel Monte Carlo approaches for parabolic SPDEs in the literature \cite{BarthLangSchwab,GilesReisinger}, the highly singular nature of \eqref{DK} equation prevents strong (pathwise) convergence of numerical solutions in the limit of arbitrarily fine discretizations. Instead, convergence takes place only up to a minimal $\underline{h}\gg N^{-1/d}$, below which the sequence begins to diverge in $C^0$. For an illustration of this lack of regularity and convergence of solutions, we refer to Figures~\ref{trajectory_snapshots_finest}--\ref{init_final_noisy_profile_rough}.
Moreover, while microscopic fluctuations in \eqref{DK} may be drastic, the macroscopic impact of fluctuations in \eqref{DK} -- i.e., the impact of fluctuations on weighted spatial averages like $\int \rho^t(x) \varphi(x) \m x$ -- is rather small due to the $N^{-1/2}$ prefactor of the noise. This discrepancy is directly related to the singular nature of the SPDE \eqref{DK}. In particular, this forces us to work with stochastically weak convergence estimates -- i.e., convergence estimates for the law of distributions -- to achieve the MLMC cost bound \eqref{MLMC_complexity}.

As with any MLMC method, we need to provide suitable bounds for the systematic error $|\mean{P_\ell- Q}|$ and the cross-level variance $Var[P_{\ell}-P_{\ell-1}]$ in order to have a cost reduction. Specifically, we have the systematic error bound 
 \begin{align}\label{systematic_error_bound}
|\mean{P_\ell - Q}| \lesssim (Nh^d_\ell)^{-1} + h^2_{\ell}
\end{align} 
(which is a relatively straightforward consequence of results in \cite{cornalba2021dean}) and a non-trivial, similarly-looking bound for $Var[P_{\ell}-P_{\ell-1}]$ (cfr. Proposition \ref{PropBoundVarLevels}), which reads
\begin{align}\label{informal_Var_bound}
Var[P_\ell - P_{\ell-1}]
\lesssim (Nh^d_{\ell})^{-1} + 
h_{\ell}^{\beta},
\end{align}
where $\beta$ is defined in \eqref{informal_lower_bound_density}.
Apart from some polynomial numerical error in $h_{\ell}$, both of the bounds \eqref{systematic_error_bound}--\eqref{informal_Var_bound} feature the inverse of the average particle density $Nh^d_\ell$: crucially, these bounds are geometrically decaying in the level $\ell$ -- and therefore conform to the assumptions of the general MLMC complexity theorem \cite[Theorem 2.1]{GilesReview} -- so long as the average particle density is sufficiently large on all levels required to achieve the systematic error $O(\varepsilon)$.
The presence of the term $(Nh^d_\ell)^{-1}$ in \eqref{systematic_error_bound}--\eqref{informal_Var_bound} is rooted in the singular nature of the Dean--Kawasaki equation, as discussed above.
However, and most importantly, having a large density requirement is not a limitation, but is in fact absolutely natural: it corresponds to a regime in which our SPDE models are computationally more efficient to simulate than the underlying particle systems.

Finally, in addition to the general complexity bound \eqref{MLMC_complexity}, we also show in Proposition \ref{PropVarFactor} that our MLMC method achieves a variance reduction factor $F$ (with respect to the standard MC method) given by 
\begin{align}\label{factor_var_red}
F \propto 
\left\{
\begin{array}{ll}
\left[(Nh^d_{\min})^{-1} + h^{(d+2)\wedge 4}_{\min} L^*\right]^{-1}|\log(h_{\min})|^{-1},  & \mbox{for Fourier coupling, cfr \&\ref{secFourierCoupling}} \\
\left[(Nh^d_{\min})^{-1} + h^2_{\min}\right]^{-1}|\log(h_{\min})|^{-1}, & \mbox{for NN coupling, cfr \&\ref{secRM}} \\
\end{array}
\right.
\end{align}
where we have set
$L^* = 1$ if $d\neq 2$, and $L^* = |\log(h_{\min})|$ if $d=2$.
\begin{remark}
While the bound \eqref{factor_var_red} is closely tied to the results of Theorem \ref{main_res}, we choose to present it on its own: the reason for this is that \eqref{factor_var_red} shows an explicit and interpretable dependence on the average particle density $Nh^d_{\min}$. 
\end{remark}
\begin{remark}
For simplicity, we focus on the non-interacting case $V\equiv 0$ in \eqref{DK} (in this case, $\bar\rho^t \equiv \mean{\empmeas[t]}$), although we expect that with the ideas of \cite{CornalbaFischerIngmannsRaithel} one could generalize the result to interacting particles. Moreover, we limit ourselves to the analysis of a finite difference scheme; due to the lack of regularity of solutions (see Figure~\ref{init_final_noisy_profile_rough} for a plot of a sample path), a finite element scheme would not offer compelling advantages.
\end{remark}

\subsection{Structure of the paper}
Notation is introduced in Section \ref{DK_Setup} (discrete Dean--Kawasaki model) and Section \ref{MLMC_sec} (Multi Level Monte Carlo setup). Assumptions, statements of main results and proofs of all main results are given in Sections \ref{sec_assumptions}, \ref{sec_Main_Result}, and  \ref{sec_proofs} respectively. Numerical results are discussed in Section \ref{num_sim}. Final considerations are given in Section \ref{summary_section}.

\section{Basic notation}\label{DK_Setup}

We work with uniformly spaced grids on the spatial domain $\domain$ of the type
$
\Grid{h}{d}:= \{-\pi ; -\pi + h; \dots; \pi - h\}^d,
$
for any compatible mesh-size $h>0$.
We denote by 
$
(f_h,g_h)_h := h^d\sum_{x\in \Grid{h}{d}}{f_h(x)\cdot g_h(x)}
$
(respectively, $\|\cdot\|_h$) the inner product (respectively, the norm) on $L^2(\Grid{h}{d})$. 
We denote by $\Ih\colon C(\domain)\rightarrow L^2(\Grid{h}{d})$ the pointwise interpolation operator.
Furthermore, $\Delta_h$, $\nabla_h$ denote discrete counterparts of Laplace and gradient operators (precise details are given later on). As for time, we work on a fixed interval $[0,T]$, and consider uniform partitions of the type
$
\Tgrid[\Dt] = \{0,\Dt,\dots, T\},
$
for a compatible time-step $\Dt>0$.

Whenever we want to provide a lower (respectively, upper) bound which holds up to a constant depending on a specific parameter $a$, we write $\gtrsim_a$ (respectively, $\lesssim_a$).

\begin{definition}[Space-time discretised Dean--Kawasaki model approximating an underlying system of $N$ non-interacting Brownian particles in $\domain$]\label{def_discrete_DK}
We denote by $\rho_{h,\Dt}$ the solution to the following finite-difference the Dean--Kawasaki equation: 
%dynamics approximating the dynamics of $N$ independent Brownian motions in $\domain$. Namely 
\begin{align}\label{FullyDiscreteDK}
\rho^{m\Dt}_{h,\Dt} & = \rho^{(m-j)\Dt}_{h,\Dt} + \Dt\sum_{j=0}^{1}{b_j \frac{1}{2}\Delta_h \rho^{(m-j)\Dt}_{h,\Dt}} \nonumber\\
& \quad \quad + N^{-1/2}\nabla_h \cdot \left(\sum_{\substack{(y,r)\in \\  (\Grid{h}{d},\{1,\dots,d\})}}{\!\!\sqrt{\left[\rho_{h,\Dt}^{(m-1)\Dt}\right]^+}e^d_{h,y,r}}\,\Delta\beta^{m-1}_{(y,r)}\right),
\end{align}
started from $\rho_{h,\Dt}^{0} = \rho_{h,0}$, where $[y]^+ := \max\{y;0\}$, and where 
\begin{itemize}
  \item $\Delta\beta^{m-1}_{(y,r)} := \beta_{(y,r)}(m\Dt) - \beta_{(y,r)}((m-1)\Dt)$ are increments for a family of independent Brownian motions $\{\beta_{(y,r)}\}_{(y,r) \in (\Grid{h}{d},\{1,\dots,d\})}$;
  \item $\rho_{h, 0}$ is a suitable -- possibly random -- initial profile;
  \item $\{e^{d}_{h, y, r}\}_{(y,r) \in (\Grid{h}{d},\{1,\dots,d\})}$ is the orthonormal basis of $[L^{2}(G_{h,d})]^d$ given by 
$
    e^{d}_{h,y, r}(z) := h^{-d/2} \mathbf{1}_{y = z} f_{r},
$
  where $\{f_{r}\}_{r=1}^{d}$ is the canonical basis of $\mathbb{R}^d$.%, and 
%  \begin{align*}
%  \delta_{y = z} := 
%  \left\{
%\begin{array}{rl}
% 1 & \mbox{if } y = z \\
% 0 & \mbox{otherwise }
%\end{array}
%\right.
%  \end{align*}
  \item $b_0, b_1$ are fixed, non-negative weights such that $b_0 + b_1 = 1$. These weights determine the specific nature of the (one-step) time discretisation scheme.
\end{itemize} 
\end{definition}

\section{MLMC framework}\label{MLMC_sec}
The goal is to compute statistical properties of the particle density fluctuations. Specifically, we aim to compute the expected value of random variables of the form \eqref{particle_fluct}
\begin{align*}%\label{particle_fluct}
Q = \normalsize\psi\left(N^{1/2}\int (\mu^T_N - \overline{\rho}^T)(x)\varphi(x)\m x  \right)
\end{align*}
for sufficiently regular test functions $\psi,\varphi$. In order to approximate $\mean{Q}$, we can compute the expected value of the quantity
\begin{align}\label{DiscrFluctuations}
P := \psi\left(N^{1/2}\left(\rho^T_{h,\Dt} - \mean{\rho^T_{h,\Dt}},\mathcal{I}_{h}\varphi\right)_{h}\right),
\end{align}
where $\rho_{h,\Dt}$ is the solution to the discrete Dean--Kawasaki model \eqref{def_discrete_DK} with space (respectively, time) discretisation parameter $h$ (respectively, $\tau$).
We think of the parameters $(h,\tau)$ in the random variable $P$ introduced in \eqref{DiscrFluctuations} as being the smallest in a sequence of space/time parameters (so, effectively, $h = h_{\min}, \tau = \tau_{\min}$): therefore, we   
 construct a MLMC scheme for approximating 
$\mean{P}$ by using the auxiliary quantities 
\begin{align}\label{eqn_MLMCLevels}
P_{\lev} & := \psi\left(N^{1/2}\left(\rho^T_{\hl,\Dtl} - {\rhob[\hl,\Dtl]^T},\mathcal{I}_{\hl}\varphi\right)_{\hl}\right),
\end{align}
on levels $\lev = 0,\dots, L$, and for some $\{(\hl , \Dt_{\lev}) \}_{\ell = 0}^{L}$. Since the deterministic dynamics of \eqref{FullyDiscreteDK} is just a finite-difference approximation of the Laplacian operator, it is natural to define the sequences $\{\hl\}_{\ell = 0}^{L}$ and $\{\Dt_{\lev}\}_{\lev=0}^{L}$ in terms of a Courant-Friedrichs-Lewy (CFL for short) ratio.
Namely, we choose $\{\hl\}_{\ell = 0}^{L}$ and $\{\Dt_{\lev}\}_{\lev=0}^{L}$ to be geometric progressions with fixed CFL ratio across levels. That is, we assume
\begin{align}\label{cfl_preserved}
\Dt_{\lev}/\hl^2 = \cfl,\qquad \forall \lev \in \{0,\dots,L\},
\end{align}
for some fixed $\cfl>0$.
The ratio of the sequence $\{\hl\}_{\ell = 0}^{L}$ (which, by \eqref{cfl_preserved}, also uniquely determines the ratio of the sequence $\{\Dt_{\lev}\}_{\lev=0}^{L}$) will depend on the specific choice of coupling across levels, as discussed in Subsections \ref{secFourierCoupling} and \ref{secRM} below. 
\begin{remark}
Upon imposing sensible assumptions, there would be no harm in also letting the ratios $\Dt_{\lev}/\hl^2$ in \eqref{cfl_preserved} depend on the level $\lev$. We do not do this, as this would entail further notational burden.
\end{remark}
%\begin{align}\label{hs_and_ts}
%(\hl, \Dt_{\lev}) & := \left( h_{\min} \cdot \sfactor^{L-\lev}, \tau_{\min} \cdot \tfactor^{L-\lev}\right), \qquad \lev = 0,\dots, L,
%\end{align}
%for some $\sfactor,\tfactor\in \mathbb{N}$ to be specified. 

With this notation, $P_{L} = P$, $(h_{\min}, \Dt_{\min}) = (h_L, \Dt_L)$, and $(h_{\max}, \Dt_{\max}) = (h_{0}, \Dt_{0})$.
%We use $\cfl$ to denote the CFL ratio $\cfl := \Dt/h^2$ (obvious extensions for subscripts $\lev,\min,\max$). 
When there is no ambiguity, we may choose to simplify the notation related to the level $\lev$: for instance, we abbreviate 
$
\psi(N^{1/2}(\rho^T_{\lev} - {\rhob[\lev]^T},\varphi_{\lev})_{\lev}) \equiv \psi(N^{1/2}(\rho^T_{\hl,\Dtl} - {\rhob[\hl,\Dtl]^T},\mathcal{I}_{\hl}\varphi)_{\hl}).
$

We consider two alternative ways to couple the noise on consecutive levels $\lev-1,\lev$. 

\subsection{First coupling: Fourier coupling}\label{secFourierCoupling}
We split the spatial grid in three at each level increment (i.e., $h_{\lev - 1} = 3h_{\lev}$).
We write the stochastic noise of \eqref{FullyDiscreteDK} using a Fourier expansion in space, and standard coupling in time. 
In words, at the coarse level $\lev-1$:
\begin{itemize}
\item all Fourier frequencies coming from 
the finer level $\lev$ are reused, and 
\item noise increments are obtained by summing up noise increments at level $\lev$. 
\end{itemize}
Explicitly, denoting the set of frequencies
$
\IndSet{\lev} := \{-\pi h_{\lev}^{-1}; -\pi h_{\lev}^{-1} + 1; \dots; \pi h_{\lev}^{-1} - 1\}^d,
$
and setting $A := \{-\pi h_{\lev}^{-1};0\}^d$, 
the noise increments at levels $\lev-1$ and $\lev$ are given by 
\begin{align*}
& W_{\lev}(x,t + \tau_{\lev}) - W_{\lev}(x,t)\\ 
& \quad := \sum_{(\xi,r) \in (\IndSet{\lev},\{1,\dots,d\})} \Four{\xi}{x}f_r (\beta_{\xi,r}(t + \tau_{\lev}) - \beta_{\xi,r}(t )),\quad \mbox{for }(x,t)\in  \Grid{\lev}{d}\times\Tgrid[\lev],\nonumber\\
& W_{\lev-1}(x,t+\tau_{\lev-1}) - W_{\lev-1}(x,t) \\
& \quad := \sum_{j=0}^{\tfactor-1}\sum_{(\xi,r) \in (\IndSet{\lev-1},\{1,\dots,d\})} \Four{\xi}{x}f_r\left[\beta_{\xi,r}(t + (j+1)\tau_{\lev}) - \beta_{\xi,r}(t + j\tau_{\lev})\right],\nonumber\\
& \quad \quad \quad \mbox{for }(x,t)\in  \Grid{\lev-1}{d}\times\Tgrid[\lev-1],\nonumber
\end{align*}
where 
$$
\left\{(\beta_{\xi,r},\beta_{-\xi,r})\right\}_{(\xi,r) \in (\IndSet{\lev-1}\setminus A,\{1,\dots,d\})} \cup \{\beta_{\xi,r}\}_{(\xi,r) \in (A,\{1,\dots,d\})} =: \mathcal{B}_{\IndSet{\lev-1}\setminus A} \cup \mathcal{B}_{A}
$$ 
is an independent family with complex conjugate couples of Brownian motions for the set $\mathcal{B}_{\IndSet{\lev-1}\setminus A}$, and standard real-valued Brownian motions for the set $\mathcal{B}_{A}$.

\begin{remark}\label{distribution_W_ell}
Note that, for each $\ell$, the noise increments $W_\ell$ have the same distribution as the noise increments of the noise prescribed in \eqref{FullyDiscreteDK} for correspoding $h_\ell,\tau_\ell$, irrespective of whether $W_\ell$ is associated with the ``fine" or the ``coarse" discretisation in a given pair of discretisations. In particular, this ensures that the Multilevel Monte Carlo estimator for $P$ is unbiased: specifically, we have that 
$
\mean{P_L} = \mean{P_{0}} + \sum_{\ell= 1}^{L}{\mean{P_\ell - P_{\ell-1}}}.
$
\end{remark}

\subsection{Second coupling: Right-Most Nearest Neighbours (NN) coupling}\label{secRM}
 While the coupling in time is the same as in the previous case, coupling in space is obtained by summing up noise increments in a ``Right-Most" neighbourhood of each spatial point. Specifically, splitting the spatial grid in two at each level increment (i.e., $h_{\lev-1} = 2h_{\lev}$), if we set $B^{\rightarrow}_{\lev-1}(x):=\{x+h_\lev v\colon v \in \{0;1\}^d\}$ for each $x\in \Grid{\lev-1}{d}$, we define
\begin{align*}
& W_{\lev}(x,t + \tau_{\lev}) - W_{\lev}(x,t) \\
& \quad := \sum_{(y,r) \in  (\Grid{\lev-1}{d},\{1,\dots,d\})} e^d_{h_{\lev},y, r}(x)(\beta_{y,r}(t + \tau_{\lev}) - \beta_{y,r}(t )), \quad \mbox{for }(x,t)\in  \Grid{\lev}{d}\times\Tgrid[\lev],\nonumber\\
& W_{\lev-1}(x,t+\tau_{\lev-1}) - W_{\lev-1}(x,t) \\
& \quad := \sum_{(y,r) \in  (\Grid{\lev-1}{d},\{1,\dots,d\})}  e^d_{h_{\lev-1},y, r}(x) \sum_{j=0}^{\tfactor-1}(\tilde{\beta}_{y,r}(t + (j+1)\tau_{\lev}) - \tilde{\beta}_{y,r}(t + j\tau_{\lev})), \nonumber\\
& \quad \quad \quad \mbox{for }(x,t)\in  \Grid{\lev-1}{d}\times\Tgrid[\lev-1],\nonumber
\end{align*}
where $\tilde{\beta}_{y,r} := 2^{-d/2}\sum_{z\in B^{\rightarrow}_{\lev-1}(y)}{\beta_{z,r}}$, with $\{\beta_{\lev,z,r}\}_{(z,r) \in (\Grid{\lev-1}{d},\{1,\dots,d\})}$ being standard independent Brownian motions. 

The contents of Remark \ref{distribution_W_ell} apply verbatim also for this coupling.

\begin{remark}\label{avoid_right_bias}

As suggested by one of the referees, the coupling in Subsection \ref{secRM} could, alternatively, be set up in a symmetric way:  
this would entail keeping the coupling as is, but using the shifted grid points $\{-\pi + h/2;-\pi+(3/2)h;\dots;\pi-h/2\}^d$.

\end{remark}

\section{Assumptions}\label{sec_assumptions} 
We work under the following assumptions.
\begin{ass}[Parameter scaling]\label{N_h_scaling} Let $L\in\mathbb{N}$ be the number of levels being considered, and let $h_{\min}$ be the smallest associated grid size. 
We choose $N$ to be sufficiently large so that the following scaling holds
\begin{align}\label{ScalingNandH}
Nh_{\ell}^d  & \gtrsim h_{\ell}^{-\beta} (1\vee \cfl^{\ths}) \left| \log{h_{\ell}} \right|, \qquad \ell \in \{0,\dots,L\}, 
\end{align}
where $\cfl := \tau_\ell/h^2_\ell$ and where 
\begin{align*}
\beta :=  
\left\{
\begin{array}{ll}
4,  & \mbox{for Fourier coupling, cfr \&\ref{secFourierCoupling}} \\
2, & \mbox{for NN coupling, cfr \&\ref{secRM}} \\
\end{array}
\right.
\end{align*}
\end{ass}

\begin{ass}[Continuous Mean Field Limit]\label{reg_mfl}
We assume the continuous mean field limit $\rhobr^t$ to be strictly positive and bounded ($0<\rho_{\min} \leq \overline{\rho}^t(x) \leq \rho_{\max} $), and satisfying $\rhobr^t\in C^{R+3}$ for some $R > 4 + d/2$ in the case of Fourier coupling (Subsection \ref{secFourierCoupling}), or $\rhobr^t \in C^{2}$ in the case of Right-Most Nearest Neighbours coupling (Subsection \ref{secRM}).
\end{ass}

\begin{ass}[Discrete Dean--Kawasaki dynamics]\label{ass_stability} With regards to \eqref{FullyDiscreteDK}: 
i) the operators $\Delta_h$ and $\nabla_h$ are second-order finite difference operators in the spatial discretisation; 
ii) the couples $\{h_\lev, \tau_\lev\}_\ell$ are such that the noise-less version of the scheme \eqref{FullyDiscreteDK} (i.e., the discrete Mean Field Limit) is stable for all $\lev$, and satisfies the discrete maximum principle.
\end{ass}

\begin{ass}[Test functions]\label{reg_test_func} For $\psi,\varphi$ introduced in \eqref{DiscrFluctuations}, we require that 
\begin{align}\label{PolGrowth1stDerPsi}
\sup_{z\in\mathbb{R}^K}|\nabla\psi(z)|/(1+|z|^r) \leq C, \qquad\mbox{for some }r>0,
\end{align}
and that $\varphi \in C^{R+4}$ for some $R > 4 + d/2$ in the case of Fourier coupling (Subsection \ref{secFourierCoupling}), or $\varphi \in C^{3}$ in the case of Right-Most Nearest Neighbours coupling (Subsection \ref{secRM}). Furthermore, we use the notation $\phi^{t}$ (respectively, $\phi^t_{\lev}$) to indicate the solution to the continuous (respectively, discrete) backwards heat equation ending at $\phi^T = \varphi$ (respectively, at $\phi_\ell^T = \mathcal{I}_\lev\varphi$).
\end{ass}

\begin{ass}[Initial datum]\label{init_datum}
We set the initial condition of the discrete mean field limit $\rhobr^0_{\lev} := \mathcal{I}_{h_\lev} \overline{\rho}^0$.
We require the following bound for large fluctuations
\begin{align}\label{InftyRhobr0}
\mean{\|\rhor^0_{\lev} - \rhobr^0_{\lev}\|^j_{\infty}} \lesssim j^j \left(N^{-1/2}h_{\lev}^{-d/2}\right)^j,\qquad j\in\mathbb{N},
\end{align}
see for instance \cite[Theorem 3.4--3.5]{chung2006concentration} and \cite{mcdiarmid1998concentration}, 
as well as the fluctuation bound
\begin{align}\label{init_h_accuracy}
\mean{\left| N^{1/2}\left(\rho^{0}_{\lev} - {\ov{\rho}_\lev^{0} },\phi_{\lev}^{0}\right)_{\lev} - N^{1/2}\int(\empmeas[0] - {\overline{\rho}^{0} })(x)\phi^{0}(x)\m x \right|^4}^{1/4} \lesssim C(\varphi) h_{\lev}^2,
\end{align}
see also \cite[Bound (3.10)]{CornalbaFischerIngmannsRaithel}. 
\end{ass}

\begin{remark}\label{rho_min_max_realistic}
The existence of $\rho_{\min}$ and $\rho_{\max}$ in Assumption \ref{reg_mfl} is guaranteed by the maximum principle for the continuous heat equation, provided we start from a strictly positive initial datum $\overline{\rho}_0$. The same $\rho_{\min}$ and $\rho_{\max}$ then also bound the discrete Mean Field Limit $\rho_{h,\tau}$ thanks to Assumption \ref{ass_stability}.
\end{remark}

\begin{remark}
%The scaling \eqref{ScalingNandH} \todo{update remark}imposes that the average particle density per grid cell (i.e., $Nh_{\ell}^d$) exceed a given threshold. The role of the terms $h_{\ell}^{-\alpha}$ and $(1\vee \cfl^{\ths}_{\ell})$ will become apparent in due course. 
Assumption \ref{init_datum} is related to a fairly standard concentration inequality, see also similar discussions in \cite{cornalba2021dean,CornalbaFischerIngmannsRaithel}.
\end{remark}

\begin{remark}
If we considered the case $V\neq 0$ (which is not treated in this paper), then we would need to require strict positivity of the continouos mean-field limit (Assumption \ref{reg_mfl}) and guarantee strict positivity of the discrete mean-field limit (Assumption \ref{ass_stability}).
\end{remark}

\section{Cross-level variance bound, and variance reduction result}\label{sec_Main_Result}

The main result of this paper (MLMC complexity, Theorem \ref{main_res}) and the variance reduction results (Proposition \ref{PropVarFactor}) rely heavily on the following Proposition.

\begin{proposition}[Bounding $Var\left\{P_\ell - P_{\ell - 1}\right\}$]\label{PropBoundVarLevels}
Assume the validity of Assumptions \ref{N_h_scaling}, \ref{ass_stability}, \ref{reg_mfl}, \ref{reg_test_func}, \ref{init_datum}, and set $\cfl := \Dt_\lev/h^2_\lev$.
We have the estimate
\begin{align}\label{BoundVarConsecLevels}
&Var[P_\ell - P_{\ell - 1}]\\
& \quad \lesssim_\varphi \min\!\left\{\exp\left(-\frac{C\rho_{\min}}{\rho_{\max}^{1/2}} \!\Big[\frac{Nh_{\lev}^{d}}{\cfl^\ths \vee 1}\Big]^{\frac{1}{2}}\right)\!\Big[N^{-1}\hl^{-d}(\cfl^{\ths} \vee 1)\Big]^{\frac{1}{2}} + \rho_{\min}^{-1} \Big[N^{-1}\hl^{-d}(\cfl^{\ths} \vee 1)\Big]; \right.\nonumber\\
& \qquad \qquad \quad \left.\Big[N^{-1}\hl^{-d}(\cfl^{\ths} \vee 1)\Big]^{\frac{1}{2}}\right\} + Err_{num} := Err_{mod} + Err_{num}, \nonumber
\end{align}
where
\begin{align*}
Err_{num} \lesssim
\left\{
\begin{array}{rl}\!
\min[\rho_{\min}^{-1}(\hl^4 + \Dt^2_{\lev}); \hl^2 + \Dt_{\lev}] + C(\rho_{\min})[ \hl^4 + \Dt_{\lev}^2 ], & \mbox{for Fourier coupling,}\\
\min\{\rho_{\min}^{-1}(\hl^2 + \Dt_{\lev}); \hl + \Dt_{\lev}\}, & \mbox{for R-M coupling.}
\end{array}
\right.
\end{align*}
\end{proposition}

\begin{remark}\label{error_naming}
The name $Err_{mod}$ above simply reflects the fact the such an error is -- morally -- of modelling type (as it is directly related to the average particle density). On the other hand, 
$Err_{num}$ is a purely numerical error.
\end{remark}

\begin{remark}
Bound \eqref{BoundVarConsecLevels} takes different forms, depending on the choice of coupling, on which of the arguments in the minima is smaller. As we do not want to unnecessarily overcomplicate the statement of our main Thorem \ref{main_res} (which stems directly from Proposition \ref{PropBoundVarLevels}), we limit ourselves to stating Proposition \ref{PropVarFactor} in two relevant subset of cases, one for each coupling.
\end{remark}

%\begin{theorem}[MLMC complexity]\label{main_res}
%\todo{fill in}
%\end{theorem}

\begin{proposition}[Variance reduction]\label{PropVarFactor}
Assume the validity of Assumptions \ref{reg_mfl}--\ref{ass_stability}--\ref{reg_test_func}--\ref{init_datum}.  
Let $M_\lev$, $\lev \in \{0,\dots,L\}$ be the number of samples of $P_\lev - P_{\lev - 1}$ (we understand $P_{-1}\equiv  0$), and let $M_{0}$ be be the number of samples of $P_{0}$. 
Choose
\begin{align}\label{NumberSamplesForLevels}
M_{\lev} := (h_{\ell}/h_{\min})^d \cdot (\Dt_{\ell}/\Dt_{\min}) 
= (h_{\ell}/h_{\min})^{d+2}  
\end{align}
and $L \propto |\log(h_{\min})|$. Assume $\overline{\rho}$ to be regular enough so that the order in the $h$-bound for $Err_{num}$ in \eqref{BoundVarConsecLevels} is the highest of the two available (i.e., 4 for Fourier coupling, 2 for Right-Most Nearest Neighbours coupling). 
Then the variance of the Multilevel Monte Carlo estimator 
\begin{align}\label{MLMCEstimator}
\EstML := M^{-1}_{0}\sum_{i=1}^{M_{0}}P_{0,(i)} + \sum_{\lev = 1}^{L}{M^{-1}_{\lev}\sum_{i=1}^{M_{\lev}}(P_{\lev,(i)} - P_{\lev - 1,(i)})}
\end{align} 
satisfies the bound
\begin{align}\label{OptimalVarBoundMLMC}
Var\left[ \EstML \right] \lesssim 
\left\{
\begin{array}{ll}
 (Nh^d_{\min})^{-1} + h_{\min}^{(d+2)\wedge 4} L^\star, & \mbox{for Fourier coupling, cfr \&\ref{secFourierCoupling}} \\
 (Nh_{\min})^{-1} + h_{\min}^{2}, & \mbox{for NN coupling, cfr \&\ref{secRM}}
\end{array}
\right.
\end{align}
where we have set
\begin{align}\label{L_star}
L^*:= 
\left\{
\begin{array}{ll}
1,  & \mbox{if }d\neq 2, \\
|\log(h_{\min})|, & \mbox{if }d= 2, \\
\end{array}
\right.
\end{align}
while carrying a total computational cost 
\begin{align}\label{CostMLMC}
Cost_{tot}(\EstML) \propto h_{\min}^{-(d+2)} \cdot|\log(h_{\min})|.
\end{align}
As a result of \eqref{OptimalVarBoundMLMC}--\eqref{CostMLMC}, using the MLMC estimator $\EstML$ defined in \eqref{MLMCEstimator} to approximate $\mean{\psi\left(N^{1/2}\int (\mu^T_N - \overline{\rho}^T)(x)\varphi(x)\m x  \right)}$ (see \eqref{particle_fluct}) -- as opposed to a standard MC estimator on the finest scale $(h_{\min}, \Dt_{\min})$ -- grants a variance reduction factor
\begin{align}\label{GainFactor}
F \propto 
\left\{
\begin{array}{ll}
\left[(Nh^d_{\min})^{-1} + h^{(d+2)\wedge 4}_{\min} L^*\right]^{-1}|\!\log(h_{\min})|^{-1},  & \mbox{for Fourier coupling, cfr \&\ref{secFourierCoupling}} \\
\left[(Nh^d_{\min})^{-1} + h^2_{\min}\right]^{-1}|\log(h_{\min})|^{-1}, & \mbox{for NN coupling, cfr \&\ref{secRM}} \\
\end{array}
\right.
\end{align}
\end{proposition}

\begin{remark}\label{remark_choice_samples}
Although the choice of samples in \eqref{NumberSamplesForLevels} amounts to a slightly less efficient MLMC estimator than what we would otherwise have if we followed standard cost optimisation  (see Algorithm \ref{algorithm} below), it however enables us to obtain the variance reduction factor \eqref{GainFactor} with a short and intuitive proof. The standard cost optimisation procedure is instead followed when producing the computational results associated with Theorem \ref{main_res} (see \cite{GilesReview} and Algorithm \ref{algorithm} below). 
\end{remark}

\section{Proofs of Proposition \ref{PropBoundVarLevels}, Theorem \ref{main_res}, and Proposition \ref{PropVarFactor}}\label{sec_proofs}

The following Lemma -- whose proof is deferred to Appendix \ref{proof_lma_ErrToMFL} -- is needed for proving Proposition \ref{PropBoundVarLevels}.

\begin{lemma}\label{lma_ErrToMFL}
 Under Assumptions \ref{N_h_scaling} and \ref{init_datum} we have
\begin{align}\label{eqn_ErrToMFL}
&\mathbb{P}\left[\sup_{t\in \Tgrid[\Dt]} \|\rho^{t}_{h,\tau}-\ov{\rho}_{h,\tau}^t\|_{\infty} \geq \frac{B\rho_{\min}}{4} \right] \lesssim \exp\left(-\frac{C\rho_{\min}}{\rho_{\max}^{1/2}}\left[\frac{Nh^{d}}{\cfl^\ths \vee 1}\right]^{1/2}\!\!\!\frac{B}{\sqrt{B+1}}\right)
\end{align}
for any $B\geq 0$, and where $\cfl = \tau/h^2$. Furthermore, we have the estimates
\begin{align}
\mean{\sup_{t\in \Tgrid[\Dt]} \|\rho^{t}_{h,\tau}-\ov{\rho}_{h,\tau}^t\|^j_{\infty}}^{1/j}\label{MomentBound_infty} 
& \lesssim C(j) \rho_{\max}^{1/2}(N^{-1}h^{-d})^{1/2}(\cfl^\ths \vee 1)^{1/2},\\
\mean{\left|\left(\rho^{t}_{h,\tau}-\ov{\rho}_{h,\tau}^t,\mathcal{I}_h\varphi\right)_\lev\right|^j}^{1/j} & \lesssim C(j)\rho_{\max}^{1/2}N^{-1/2}\|\varphi\|_{C^1}.\label{MomentBound}
\end{align}
\end{lemma}

\subsection{Proof of Proposition \ref{PropBoundVarLevels} with Fourier coupling}

For notational convenience, we occasionally drop the time dependence, and the vectorial notation over the $d$ components of the noise (i.e., we write $\Four{\xi}{x}$ instead of $\{\Four{\xi}{x}f_r\}_{r=1}^{d}$).

\emph{Step 1: rewriting $Var[P_\ell - P_{\ell - 1}]$}. A first-order Taylor expansion on $\psi$ gives 
\begin{align}\label{BoundVar}
Var[P_\ell - P_{\ell - 1}] & \leq \mean{\left| P_\ell - P_{\ell - 1} \right|^2} \nonumber\\
&  = \mean{\left| \psi\left(N^{1/2}\left(\rho^T_{\lev} - {\rhob[\lev]^T},\varphi_{\lev}\right)_{\lev}\right) - \psi\big(N^{1/2}\left(\rho^T_{\lev-1} - {\rhob[\lev-1]^T},\varphi_{\lev-1}\right)_{\lev-1}\big) \right|^2} \nonumber\\
&  \leq \mean{|\nabla\psi(z)|^4}^{\frac{1}{2}}\mean{\left| N^{\frac{1}{2}}\!\!\left[\left(\rho^T_{\lev} - {\rhob[\lev]^T},\varphi_{\lev}\right)_{\lev} - \left(\rho^T_{\lev-1} - {\rhob[\lev-1]^T},\varphi_{\lev-1}\right)_{\lev-1}\right] \!\right|^4}^{\frac{1}{2}}
\end{align}
for some random $z$ such that 
\begin{align}\label{BoundZ}
|z| \leq \max\left\{\left| N^{1/2}\left(\rho^T_{\lev} - {\rhob[\lev]^T},\varphi_{\lev}\right)_{\lev}\right|; \big|N^{1/2}\left(\rho^T_{\lev-1} - {\rhob[\lev-1]^T},\varphi_{\lev-1}\right)_{\lev-1}\big|\right\}.
\end{align} 

\emph{Step 2: bounding $\mean{|\nabla\psi(z)|^4}$ in \eqref{BoundVar}}. Using \eqref{PolGrowth1stDerPsi}, \eqref{BoundZ} and \eqref{MomentBound}, we get
\begin{align}\label{Bound4thPowerGradPsi}
\mean{|\nabla\psi(z)|^4} \leq C(\rho_{\max},r)\|\varphi\|^{4r}_{C^{1}}.
\end{align}

\emph{Step 3: It\^o formula for second term in right-hand-side of \eqref{BoundVar}}. Let the test functions $\phi_{\lev}$ (respectively, $\phi_{\lev -1}$) satisfy the backwards evolution \eqref{BackwardsTestOneStep} ending in $\varphi_\lev$ (respectively, in $\varphi_{\lev -1}$). Thanks to the discrete martingale property in Lemma \ref{lem_Discrete_Martingale} and a simple interpolation argument on the noise, we can define a continuous-time martingale $D^t_{\lev,\lev-1}$ such that, crucially, 
\begin{align*}
D^t_{\lev,\lev-1} = N^{1/2}\left(\rho^{t}_{\lev} - {\ov{\rho}_\lev^{t} },\phi_{\lev}^{t}\right)_{\lev} - N^{1/2}\left(\rho^{t}_{\lev-1} - {\ov{\rho}_{\lev-1}^{t} },\phi_{\lev-1}^{t}\right)_{\lev-1}\qquad \mbox{for all }t\in \mathcal{S}_{\lev-1}.
\end{align*}
Now abbreviate $\btt{\lev} := \max\{m \in \Tgrid[\Dt_\lev]: m < t \}$.
Using the continuous It\^o Lemma, the noise coupling as stated in Subsection \ref{secFourierCoupling}, and the error bound \eqref{Matrix_error}, we deduce
\begin{align}
& \mean{(D^z_{\lev,\lev-1})^4} \nonumber\\
& \quad \lesssim \mean{(D^0_{\lev,\lev-1})^4}\nonumber\\
& \quad \quad + \int_{0}^{z}{\mathbb{E}\Bigg[(D^t_{\lev,\lev-1})^2 \times \Dt^2_{\lev} + (D^t_{\lev,\lev-1})^2 \, \times }\nonumber\\
& \left. \quad \quad \quad \quad \times \Bigg\{\sum_{\xi \in \IndSet{\lev-1}}{\left|\bigg(\sqrt{\rho^{\btt{\lev},+}_{\lev}}\F{\xi},\nabla_\lev\phi^{\btt{\lev}}_\lev\bigg)_{\!\lev} - \left(\sqrt{\rho^{\btt{\lev-1},+}_{\lev-1}}\F{\xi},\nabla_{\lev-1}\phi^{\btt{\lev-1}}_{\lev-1}\right)_{\!\lev-1}\right|^2}\right. \nonumber\\
& \quad \quad \quad \quad \quad \quad \quad \quad +  \sum_{\xi \in \IndSet{\lev}\setminus\IndSet{\lev-1}}{\left|\left(\sqrt{\rho^{\btt{\lev},+}_{\lev}}\F{\xi},\nabla_\lev\phi^{\btt{\lev}}_\lev\right)_{\lev}\right|^2} \Bigg\} \Bigg] \m t \nonumber\\
& \quad =: \mean{(D^0_{\lev,\lev-1})^4} + \Dt_{\lev}^4 + \int_{0}^{z}{\mean{(D^t_{\lev,\lev-1})^2 \times \left\{ A_{\IndSet{\lev-1}} + A_{\IndSet{\lev}\setminus\IndSet{\lev-1}}\right\}}\m t},\label{ItoForFourthMoment}
\end{align}
where $(\cdot)$ in $\F{\xi}$ is reserved for the spatial variable $x$.
The inequality \eqref{ItoForFourthMoment} implies
\begin{align}
& \mean{(D^z_{\lev,\lev-1})^4} \nonumber\\
& \quad \lesssim\mean{(D^0_{\lev,\lev-1})^4}  +  \Dt_{\lev}^4 + \int_{0}^{z}{\mean{(D^t_{\lev,\lev-1})^4}}\m t + \int_{0}^{z}{\mean{ A^2_{\IndSet{\lev-1}} + A^2_{\IndSet{\lev}\setminus\IndSet{\lev-1}}} \m t}. \label{Expand4thMomentDiff}
\end{align}

\emph{Step 4: bounding $A^2_{\IndSet{\lev-1}} $ in \eqref{Expand4thMomentDiff}}. By adding and subtracting zero we get
\begin{align}\label{ExpandInBs}
A_{\IndSet{\lev-1}} & \lesssim \sum_{\xi \in \IndSet{\lev-1}}{\left|\left(\sqrt{\rho^{\btt{\lev},+}_{\lev}}\F{\xi},\nabla_\lev\phi^{\btt{\lev}}_\lev\right)_{\lev} - \left(\sqrt{\ov{\rho}^{\btt{\lev}}}\F{\xi},\nabla\phi^{\btt{\lev}}\right)_{\lev}\right|^2} \nonumber\\
& \quad + \sum_{\xi \in \IndSet{\lev}}{\left| \left(\sqrt{\ov{\rho}^{\btt{\lev}}}\F{\xi},\nabla\phi^{\btt{\lev}}\right)_{\lev} - \left(\sqrt{\ov{\rho}^{\btt{\lev-1}}}\F{\xi},\nabla\phi^{\btt{\lev-1}}\right)_{\lev} \right|^2} \nonumber\\
& \quad + \sum_{\xi \in \IndSet{\lev-1}}{\left| \left(\sqrt{\ov{\rho}^{\btt{\lev-1}}}\F{\xi},\nabla\phi^{\btt{\lev-1}}\right)_{\lev} - \left(\sqrt{\ov{\rho}^{\btt{\lev-1}}}\F{\xi},\nabla\phi^{\btt{\lev-1}}\right)_{\lev-1}\right|^2} \nonumber\\
& \quad + \sum_{\xi \in \IndSet{\lev-1}}{\left| \left(\sqrt{\rho^{\btt{\lev-1},+}_{\lev-1}}\F{\xi},\nabla_{\lev-1}\phi^{\btt{\lev-1}}_{\lev-1}\right)_{\lev-1} - \left(\sqrt{\ov{\rho}^{\btt{\lev-1}}}\F{\xi},\nabla\phi^{\btt{\lev-1}}\right)_{\lev-1}\right|^2} \nonumber\\
& =: R_{\lev} + R_{\btt{\lev},\btt{\lev -1}} + R_{\lev,\lev -1} + R_{\lev-1}.
\end{align}
Adding and subtracting zero (and dropping the dependence on $\btt{\lev}$), we obtain 
\begin{align}\label{expandRell}
R_{\lev} & \lesssim \sum_{\xi \in \IndSet{\lev-1}}{\left|\Big(\Big[\sqrt{\rho^+_{\lev}} - \sqrt{\rhob}\Big]\F{\xi},\nabla_\lev\phi_\lev\Big)_{\lev} - \left(\sqrt{\rhob}\F{\xi},[\nabla\phi - \nabla_\lev\phi_\lev]\right)_{\lev}\right|^2} \\
& \lesssim \Big\| \Big[ \sqrt{\rho^+_{\lev}} - \sqrt{\rhob} \Big] \nabla_\lev\phi_\lev\Big\|^2_{\lev} + \left\| \sqrt{\rhob}[\nabla\phi - \nabla_\lev\phi_\lev] \right\|^2_{\lev} \nonumber\\
& \lesssim \mathbf{1}_{\inf \rho_{\lev}\leq {K}}\Big\| \Big[ \sqrt{\rho^+_{\lev}} - \sqrt{\rhob} \Big] \nabla_\lev\phi_\lev\Big\|^2_{\lev} + \mathbf{1}_{\inf \rho_{\lev}> {K}}\cdot\rho_{\min}^{-1}\left\| \left[ \rho_{\lev} - \rhob \right] \nabla_\lev\phi_\lev\right\|^2_{\lev} \nonumber\\
& \quad + \left\| \sqrt{\rhob}[\nabla\phi - \nabla_\lev\phi_\lev] \right\|^2_{\lev} =: R_{\lev,1} + R_{\lev,2} + R_{\lev,3}\nonumber,
\end{align}
with $K$ to be specified later.
An analogous bound holds for $R_{\lev-1}$.

In order to treat $R_{\lev,\lev-1}$ (we drop the time dependence $\btt{\lev-1}$ for this term), we abbreviate $g := \sqrt{\rhobr} \nabla \phi$ and perform the rewriting
\begin{align}\label{DiffBetweenLevels}
& \left(\sqrt{\rhob}\F{\xi},\nabla\phi\right)_{\lev} - \left(\sqrt{\rhob}\F{\xi},\nabla\phi\right)_{\lev - 1} \nonumber\\
& \quad = \sum_{x\in \Grid{\lev-1}{d}}{h_{\lev - 1}^d \Bigg( \Four{\xi}{x}g(x) - \sum_{\Grid{\lev}{d} \ni y\sim x}{3^{-d}\Four{\xi}{y}(y)g(y)}\Bigg)},
\end{align}
where $y\sim x$ indicates all points $y \in \Grid{\lev}{d}$ such that $|y-x|\leq h_{\lev}$. In order to sum up over the frequencies $\xi$, we integrate by parts on the lattice $\Grid{\lev - 1}{d}$. More specifically, we rely on the increment relation
\begin{align}\label{IncrementFourier}
e^{i\xi\cdot (x + he_j)} - e^{i\xi\cdot x} = e^{i\xi\cdot x}\underbrace{(e^{i\xi_j h}-1)}_{=: P(\xi_j h)}, \qquad j\in\{1,\dots,d\}.
\end{align} 
Take $\xi\neq 0$. Using \eqref{IncrementFourier} with the same increment $h=h_{\lev-1}$ in all sums of \eqref{DiffBetweenLevels}, we perform discrete integration by parts $R$ times in direction $j$ such that $\xi_j \neq 0$, and get
\begin{align}\label{DiffBetweenLevels_2}
& \left(\sqrt{\rhob}\F{\xi},\nabla\phi\right)_{\lev} - \left(\sqrt{\rhob}\F{\xi},\nabla\phi\right)_{\lev - 1} \nonumber\\
& = \left[\frac{-1}{P(\xi_j h)}\right]^R \sum_{x\in \Grid{\lev-1}{d}}{h_{\lev - 1}^d \left( \Four{\xi}{x} \partial^{R}_{h,j}g(x) - \sum_{\Grid{\lev}{d} \ni y\sim x}{3^{-d}\Four{\xi}{y}(y)\partial^{R}_{h,j}g(y)}\right)},
\end{align}
where $\partial^{R}_{h,j}f(x) := \sum_{i=0}^{R}{(-1)^i}\binom{R}{i}f(x-ihe_j)$ is the standard backwards finite difference operator satisfying the relation
\begin{align}\label{Order1Der}
\partial^{R}_{h,j}f(x) = h^R\partial_{j}^Rf(x) + h^{R+1}C(R)\partial^{R+1}f(\zeta), \,\,\mbox{ for some }\zeta=\zeta(x)\in \domain.
\end{align} 

By using the order-two bound $| f(x+he_j) - 2f(x) + f(x-he_j)| \lesssim h^2 \|f\|_{C^2}$ in the round bracket of \eqref{DiffBetweenLevels_2} (with $f(x):= e^{i\xi\cdot x}\partial^{R}_{h,j}g(x)$), and the equivalence $\sin(z)\propto z$ for $z\in [-\pi/2,\pi/2]$, we rely on Assumptions \ref{reg_mfl}--\ref{reg_test_func} to integrate by parts $R$ times in all directions $j$ with $\xi_j \neq 0$ (with $R$ to be determined), and deduce
\begin{align}\label{change_levels}
& \left|\left(\sqrt{\rhob}\F{\xi},\nabla\phi\right)_{\lev} - \left(\sqrt{\rhob}\F{\xi},\nabla\phi\right)_{\lev - 1}\right| \nonumber\\
& \quad \leq \frac{1}{(\sum_{j=1}^{d}{|P(\xi_j h_{\lev-1})|)}^R}\sum_{y\in \Grid{\lev}{d}}{h_{\lev - 1}^d \|g\|_{C^{R+3}}}h_{\lev - 1}^{R}h_{\lev - 1}^{2}|\xi|^4 \nonumber\\
& \quad \leq |\xi|^{4-R}\|g\|_{C^{R+3}}h_{\lev - 1}^{2},
\end{align}
where we have also used \eqref{Order1Der} in the second-to-last inequality and the fact that  
$
|P(\xi_j h)| = |\sin(\xi_j h / 2)|, \forall j\in\{1,\dots,d\},
$
in the last inequality. The estimate for $\xi = 0$ is even simpler, and requires no integration by parts at all.
Provided $R>4+d/2$, \eqref{change_levels} implies that we can sum up the terms making up $R_{\lev,\lev-1}$ in \eqref{ExpandInBs} and get 
\begin{align}\label{BoundB2}
R_{\lev,\lev-1} \lesssim \|g\|^2_{C^{R+3}}h_{\lev - 1}^{4}.
\end{align}

\emph{Step 5: bounding $A^2_{\IndSet{\lev}\setminus\IndSet{\lev-1}}$ in \eqref{Expand4thMomentDiff}}. Dropping the dependence on $\btt{\lev}$, we write
\begin{align}\label{ExpandInBs_HighFreq}
A_{\IndSet{\lev}\setminus\IndSet{\lev-1}} & \lesssim \sum_{\xi \in \IndSet{\lev}\setminus\IndSet{\lev-1}}{\Big|\Big(\sqrt{\rho^+_{\lev}}\F{\xi},\nabla_\lev\phi_\lev\Big)_{\lev} - \left(\sqrt{\rhob}\F{\xi},\nabla\phi\right)_{\lev}\Big|^2} \nonumber\\
& \quad + \sum_{\xi \in \IndSet{\lev}\setminus\IndSet{\lev-1}}{\left|\left(\sqrt{\rhob}\F{\xi},\nabla\phi\right)_{\lev}\right|^2} =: \ov{R}_{\lev} + \ov{R}_{\IndSet{\lev}\setminus\IndSet{\lev-1}}.
\end{align}
The term $\ov{R}_{\lev}$ obviously has the same bound as the previously treated $R_{\lev}$.
Furthermore, as $R > 4 + d/2$, and reusing computations from \eqref{change_levels}, we get
\begin{align}\label{BoundC2}
\ov{R}_{\IndSet{\lev}\setminus\IndSet{\lev-1}} & \leq \sum_{\xi \in \IndSet{\lev}\setminus\IndSet{\lev-1}}{\frac{1}{|P(\xi h_{\lev-1})|^{2R}}\|g\|^2_{C^R}h^{2R}_{\lev - 1}} \nonumber\\
& \leq \sum_{\xi \in \IndSet{\lev}\setminus\IndSet{\lev-1}}{\frac{1}{|\xi|^{2R}}\|g\|^2_{C^R}} \leq \|g\|^2_{C^R}h_{\lev - 1}^4. 
\end{align}

\emph{Step 6: taking the expectations in \eqref{Expand4thMomentDiff}}. In order to bound $\mean{R_{\lev,1}^2 + R_{\lev,2}^2}$, we consider the best estimate originating from picking either $K=\rho_{\min}/2$ or $K=\infty$ in \eqref{expandRell}: More precisely, using
 H\"older's inequality, Lemma \ref{lma_ErrToMFL}, and the simple estimate $\|\overline{\rho} - \overline{\rho}_\lev\|_\lev\propto h_{\lev}^2 + \tau_{\lev}$, we obtain
\begin{align*}
& \mean{R_{\lev,1}^2 + R_{\lev,2}^2} \\
& \quad \lesssim_\varphi \min\!\left\{\exp\left(-\frac{C\rho_{\min}}{\rho_{\max}^{1/2}} \Big[\frac{Nh_\ell^{d}}{\cfl^\ths \vee 1}\Big]^{\frac{1}{2}}\right)\!\Big[N^{-1}\hl^{-d}(\cfl^{\ths} \vee 1)\Big] + \rho_{\min}^{-1} \Big[N^{-1}\hl^{-d}(\cfl^{\ths} \vee 1)\Big]^2\!; \right. \nonumber\\
& \qquad \qquad \qquad \Big[N^{-1}\hl^{-d}(\cfl^{\ths} \vee 1)\Big]\bigg\} + \min\{\rho_{\min}^{-2}(\hl^8 + \Dt^4_{\lev}); \hl^4 + \Dt_{\lev}^2 \}.
\end{align*}
As $\|\nabla \phi - \nabla_\lev\phi_\lev\| \propto h^{2}_\lev + \Dt_{\lev}$ (cfr. Lemma \ref{Lemma_ErrorDiscrTestFunc}), we get 
$
\mean{R_{\lev,3}^2} \leq C_{\varphi}(h_{\lev}^{8} + \Dt_{\lev}^4),
$
and an analogous bound holds for $R_{\lev-1}$. Similar arguments (involving the same thresholds $K=\rho_{\min}/2$ or $K=\infty$), a Taylor expansion, and the discrete Parseval identity grant the estimate $R_{\btt{\lev},\btt{\lev -1}} \lesssim_\varphi \min\{\rho^{-1}_{\min}\tau_\ell^2,\tau_\lev\}$.

\emph{Step 7: concluding the argument}.
Combining the estimates in \emph{Steps 4--6}, we get 
\begin{align*}
&\mean{ A^2_{\IndSet{\lev-1}} + A^2_{\IndSet{\lev}\setminus\IndSet{\lev-1}}}\\
& \quad \lesssim_\varphi \min\!\left\{\exp\left(-\frac{C\rho_{\min}}{\rho_{\max}^{1/2}} \Big[\frac{Nh_\ell^{d}}{\cfl^\ths \vee 1}\Big]^{\frac{1}{2}}\right) \Big[N^{-1}\hl^{-d}(\cfl^{\ths} \vee 1)\Big] + \rho_{\min}^{-1} \Big[N^{-1}\hl^{-d}(\cfl^{\ths} \vee 1)\Big]^2; \right.\nonumber\\
& \qquad \qquad \qquad \Big[N^{-1}\hl^{-d}(\cfl^{\ths} \vee 1)\Big]\bigg\} \\
&\qquad\quad+ \min\{\rho_{\min}^{-2}(\hl^8 + \Dt^4_{\lev}); \hl^4 + \Dt_{\lev}^2\} + C(\rho_{\min})\{ \hl^8 + \Dt_{\lev}^4 \}.
\end{align*}
Using the above inequality, \eqref{init_h_accuracy},  \eqref{BoundVar} and \eqref{Bound4thPowerGradPsi}, we apply Gronwall's Lemma in \eqref{Expand4thMomentDiff} and conclude the proof.

\begin{remark}
The proof of Proposition \ref{PropBoundVarLevels} with Right-Most Nearest Neighbours coupling is similar to the one we have provided for the Fourier coupling. The main adjustment concerns adapting the cross-variation structure (i.e., replacing the terms $A_{\IndSet{\lev-1}} + A_{\IndSet{\lev}\setminus\IndSet{\lev-1}}$ in \eqref{ItoForFourthMoment}). It is easy to see that that the term $A_{\IndSet{\lev-1}} + A_{\IndSet{\lev}\setminus\IndSet{\lev-1}}$ in \eqref{ItoForFourthMoment} is replaced by
\begin{align}\label{eqn:VarOfLvlPrelim}
& \sum_{x \in \Grid{\lev}{d}}{h^d_{\ell} \rho^+_{h_\ell}(x)|\nabla_{h_\ell}\phi_{h_\ell}(x)|^2} + \sum_{x \in \Grid{\lev-1}{d}}{h^d_{\ell-1} \rho^+_{h_{\ell-1}}(x)|\nabla_{h_{\ell-1}}\phi_{h_{\ell-1}}(x)|^2} \nonumber \\
& \quad \quad\quad -2\sum_{y \in B^{\rightarrow}_{\lev-1}(x)}\sum_{x\in \Grid{\lev-1}{d}}h_{\ell}^d\sqrt{\rho^+_{h_\ell}(y)} \nabla_{h_\ell}\phi_{h_\ell}(y)
 \cdot\sqrt{\rho^+_{h_{\ell-1}}(x)}\nabla_{h_{\ell-1}}\phi_{h_{\ell-1}}(x) \nonumber \\
& \quad = \sum_{y \in B^{\rightarrow}_{\lev-1}(x)}\sum_{x\in \Grid{\lev-1}{d}}{h^d_{\ell}\left| \sqrt{\rho^+_{h_\ell}(y)}\nabla_{h_\ell}\phi_{h_\ell}(y)  - \sqrt{\rho^+_{h_{\ell-1}}(x)}\nabla_{h_{\ell-1}}\phi_{h_{\ell-1}}(x)  \right|^2}.
\end{align}
From \eqref{eqn:VarOfLvlPrelim}, it is relatively straightforward to get the $Err_{num}$ contribution using similar techniques to those deployed in \eqref{expandRell} (i.e., choosing the best cutoff for the square-root between $K=\rho_{\min}/2$ and $K=\infty$). All other components of the estimate do not change with respect to the proof in the Fourier case.
\end{remark}
\begin{remark}\label{Quadratic_vs_Linear}
Having two separate estimates originating from two cut-off values $K=\rho_{\min}/2$ or $K=\infty$ helps providing a reliable estimate for $Err_{num}$ in \eqref{BoundVarConsecLevels}: However, these two estimates can -- in many cases -- still be substantially suboptimal, as they do not rely on the local structure of the mean field limit $\overline{\rho}$. A relevant example is discussed in our simulations in Section \ref{num_sim}.
\end{remark}

\subsection{Proof of Theorem \ref{main_res}}

Let $\kappa$ be the common ratio of the geometric sequence $(\hl)_{\ell}$.
Assumption \ref{N_h_scaling} implies that, for all levels needed to get to the minimum grid size $h^2_{\min} \lesssim \varepsilon$, we have the bound
\begin{align*}
Var[P_\ell - P_{\ell-1}] \lesssim \hl^{\beta} \propto \kappa^{-\beta \lev},
\qquad
\beta =  
\left\{
\begin{array}{ll}
4,  & \mbox{for Fourier coupling, cfr \&\ref{secFourierCoupling}} \\
2, & \mbox{for NN coupling, cfr \&\ref{secRM}} \\
\end{array}
\right.
\end{align*}
while a simple extension of the results in \cite{cornalba2021dean} give 
\begin{align}\label{geom_decay_mean}
|\mean{ P_\ell - Q }| \lesssim \hl^{2} \propto \kappa^{-\alpha \lev},\qquad \alpha =2.
\end{align}
Additionally, the computational cost for simulating a single sample of $P_{\ell}$ grows like
\begin{align}\label{cost_P_l}
Cost(P_\ell) \lesssim \hl^{-d-2} \propto \kappa^{(d+2) \lev}, \qquad \gamma = d+2.
\end{align}
Then \eqref{MLMC_complexity} follows from applying the general MLMC complexity theorem \cite[Theorem 2.1]{GilesReview}.

\subsection{Proof of Proposition \ref{PropVarFactor}}  

Since $\overline{\rho}$ is assumed to be regular, we have from Proposition \ref{PropBoundVarLevels} that $Var[P_{\ell} - P_{\ell-1}] \lesssim (Nh^d_\ell)^{-1} + h_{\ell}^\beta$, with $\beta = 4$ for the Fourier coupling and $\beta = 2$ for the NN coupling.
Using \eqref{NumberSamplesForLevels}, and denoting by $\kappa$ the common ratio of the geometric sequence $(\hl)_{\ell}$, we deduce
\begin{align}\label{VarBoundAlphaD}
Var\left[ \EstML \right] & \lesssim M^{-1}_{0} + \sum_{\lev = 1}^{L}{M_{\lev}^{-1}\left[ (N\hl^{d})^{-1} + \hl^\beta\right]} \nonumber\\
& \lesssim M^{-1}_{0} + \sum_{\lev = 1}^{L}{\left(\frac{h_{\min}}{\hl}\right)^{d+2}\left[ (N\hl^{d})^{-1} + \hl^\beta\right]} \nonumber\\
& \lesssim M^{-1}_{0} + (Nh^d_{\min})^{-1} + h^{d+2}_{\min}\sum_{\ell=1}^{L}{(h_{\min}\kappa^{L-\ell})^{\beta - d- 2}}
\end{align}

The choice $L \propto |\log(h_{\min})|$ and \eqref{VarBoundAlphaD} readily imply the variance bound \eqref{OptimalVarBoundMLMC}. 
Proving \eqref{CostMLMC} is immediate. 
Finally, the variance reduction factor \eqref{factor_var_red} can be readily deduced by dividing the MLMC variance \eqref{VarBoundAlphaD} with the variance of the standard MC estimator with the same cost as \eqref{CostMLMC}. Proposition \ref{PropVarFactor} is thus proved.

\section{Numerical simulations}\label{num_sim}

We provide numerical simulations demonstrating the validity of our theoretical results in the two dimensional case ($d=2$), and with the NN coupling.

\subsection{Types of experiments}\label{exp_type}

The experiments conducted are of two types:

\underline{\emph{Type 1}}: in this type of experiment, we provide computational results for Theorem \ref{main_res} by running the MLMC implementation as described in \cite[Algorithm 1]{GilesReview}. A minimal version of the pseudocode of this implementation is given in Algorithm \ref{algorithm}. 

In essence, to reach a given accuracy $\varepsilon$, Algorithm \ref{algorithm} recursively keeps on adding levels and computing the optimal number of samples
\begin{align}\label{opt_M_ell}
M_{\ell} = \left\lceil 2\varepsilon^{-2}\sqrt{V_\ell/C_\ell}\left(\sum_{\ell=0}^{L}{\sqrt{V_\ell C_{\ell}}}\right) \right\rceil
\end{align}
for all $\ell= 0,1,\dots,L$, where $L$ is the current maximum level, $V_{\ell} := Var[P_\ell - P_{\ell-1}]$, $\ell = 0,\dots, L$ (for notational simplicity, we always understand the notation $P_{-1}$ to mean $P_{-1}\equiv 0$) and $C_{\ell}$ is the cost of computing one sample on level $\ell$. The convergence criterion identifying the level $L$ at which the algorithm stops is given by the following robust check for the systematic error\footnote{the criterion \eqref{convergence_criterion} is the same as the one in \cite{GilesReview}, which is derived under the power law decay assumption for the systematic error on all $\ell\in\mathbb{N}$. As we do not have power law decay on all $\ell\in\mathbb{N}$, we heuristically expect good performance of this convergence criterion under the -- usual -- assumption of large average particle density.}:
\begin{align}\label{convergence_criterion}
\mbox{stop on level }L\mbox{ if }\max_{\tilde{\ell}\in\{0;1;2\}}\{2^{-\tilde{\ell}\alpha}|\mean{P_{L-\tilde{\ell}}-P_{L-\tilde{\ell}-1}}|\}/({2^{\alpha}-1})<\frac{\varepsilon}{\sqrt{2}}.
\end{align}
for $\alpha$ as in \eqref{geom_decay_mean}.
\begin{algorithm*}[h]
\caption{(see \cite[Algorithm 1]{GilesReview})}
\label{algorithm}
\begin{algorithmic}
\State Choose accuracy $\varepsilon > 0$ and $\overline{M}\in\mathbb{N}$. 
\State Start with $L = 2$, initial target of $\overline{M}$ samples on levels $\ell = 0, 1, 2$.
\While{extra samples need to be evaluated}
\State evaluate extra samples on each level
\State compute/update estimates for $V_\ell := Var[P_\ell - P_{\ell-1}]$, $\ell = 0,\dots, L$.
\State define optimal $M_\ell$, $\ell = 0,\dots, L$, according to \eqref{opt_M_ell}. 
\State test for weak convergence using \eqref{convergence_criterion}.
\If{not converged}
    \State set $L \gets L+1$
    \State initialize target $M_L$.
\EndIf

\EndWhile
\end{algorithmic}
\end{algorithm*}

\underline{\emph{Type 2}}: for this type of experiment, the number of samples $M_{\ell}$ is chosen to be proportional to a given, specified geometric progression. The variance of the resulting MLMC estimator is then evaluated against the variance of the corresponding standard MC estimator with the same computational time. A computational estimate of the variation reduction factor in Proposition \ref{PropVarFactor} is thus produced.

\subsection{Setting and specifics}\label{sec_setting}
The precise details of our setting are as follows:
\begin{itemize}
\item Space domain: we work on discretizations of $\mathbb{T}^2 = (0,2\pi)^2$.
\item Initial condition for mean-field dynamics: we use two different initial conditions. The first one
\begin{align}
\overline{\rho}_{0,reg}  = Z_{reg}^{-1}\left(1+e^{-(\sin^2(x-\pi/2) +\sin^2(y-3\pi/2))/2}/(\sqrt{2\pi})\right)
\end{align} 
-- where $Z_{reg}$ is the normalising constant -- is bounded away from zero, and with relatively low ratio $\sqrt{\rho_{\max}}/\rho_{\min}\approx 13.4$. The second initial condition 
\begin{align}
\overline{\rho}_{0,irreg}  = Z_{irreg}^{-1}e^{-(\sin^2(x-\pi/2) +\sin^2(y-3\pi/2))/(2\cdot 0.1)}%/(\sqrt{2\pi \cdot 0.1})
\end{align} 
has ultra-low density regions, and much larger ratio $\sqrt{\rho_{\max}}/\rho_{\min} \gg 10^6$. The two initial conditions are shown in Figure \ref{init_cond}.
\begin{figure}[h]
\begin{center}
\includegraphics[width=0.45\linewidth]{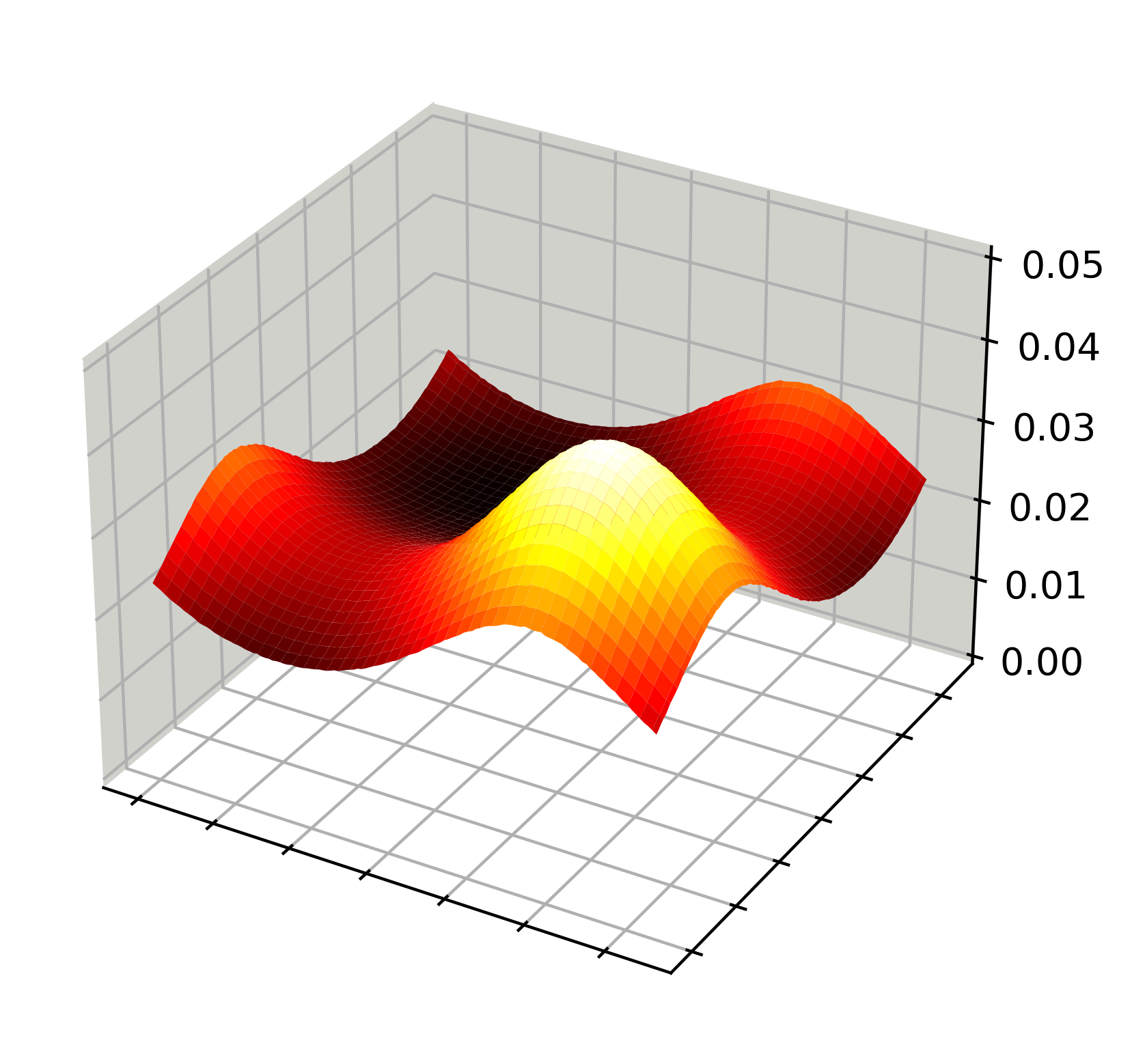}
\includegraphics[width=0.45\linewidth]{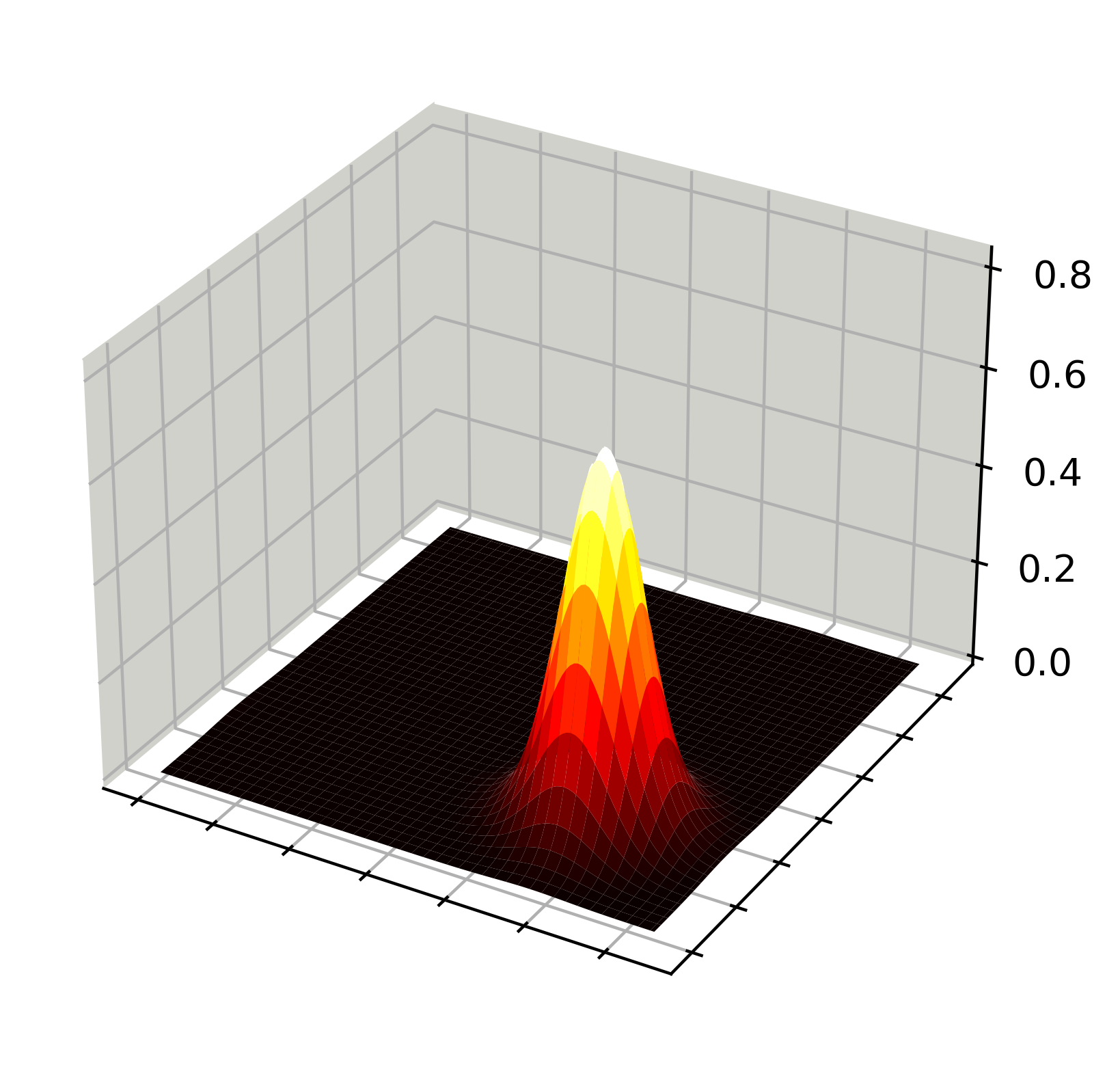}
\caption{\label{init_cond}\emph{Left}: 3d heat map of $\overline{\rho}_{0,reg} $; \emph{Right}: 3d heat map of $\overline{\rho}_{0,irreg} $}
\end{center}
\end{figure}
\item Time-stepping: we stick to a simple explicit Euler-Maruyama scheme, i.e., we pick $b_0 = 0, b_1 = 1$ in \eqref{FullyDiscreteDK}.
\item Multilevel Monte Carlo discretisation: we consider at most $ L = 5 $, and use Right-Most Nearest Neighbours coupling, see Subsection \ref{secRM}.
When $L=5$, for the finest level in space, each axis is split in $2^7= 128$ parts, i.e., 
$
h_{\min}=2\pi\cdot 2^{-7}\approx 0.05$. Furthermore, the finest level in time has time-step $t_{min}=10^{-3}$. 
Consecutive levels are designed to preserve the CFL condition: specifically, we choose
$h_{\ell-1}=2h_{\ell}$ and $ \tau_{\ell - 1} = 4\tau_{\ell}$,
so that $\cfl = \tau_{\ell}/h_{\ell}^2$ is constant over $\ell$.
Putting all together, we have
\begin{align*}
h_\ell = 2\pi \cdot 2^{-(2+\ell)}, \qquad \tau_\ell = 10^{-3}\cdot 4^{5-\ell}.
\end{align*}

Concerning the number of samples on each level:
\begin{itemize}
\item For the purpose of estimating the computational gain of the MLMC method for given accuracy $\varepsilon$ (so for the experiments of \underline{\emph{Type 1}}), the number of samples is determined by Algorithm \ref{algorithm} (i.e., by \eqref{opt_M_ell}).
\item For the purpose of validating the variance reduction factor  estimate \eqref{factor_var_red} (so for the experiments of \underline{\emph{Type 2}}), we instead run the MLMC with the pre-determined geometric progression\footnote{We chose this geometric progression for $M_{\ell}$ -- which is slightly different than that indicated in \eqref{NumberSamplesForLevels} -- for purely practical implementation reasons. Note that this choice does not affect the scaling of the variance reduction factor \eqref{factor_var_red}, and this is also verified numerically.} $M_{\ell-1} = 4M_{\ell}$. 
\end{itemize} 
% Furthermore, the number of samples scales with factor four with each level change, i.e., $M_{\ell-1} = 4M_{\ell}$

For qualitative illustration purposes, we show snapshots of a given trajectory of \eqref{FullyDiscreteDK} in Figure \ref{trajectory_snapshots_finest} (2d heat maps) and Figure \ref{init_final_noisy_profile_rough} (3d heat maps)\footnote{The graphics commands used to generate Figure \ref{trajectory_snapshots_finest} are based on the code available at: \url{https://scipython.com/book/chapter-7-matplotlib/examples/the-two-dimensional-diffusion-equation/}}.

\begin{figure}[h]
\begin{center}
\includegraphics[width=0.49\linewidth]{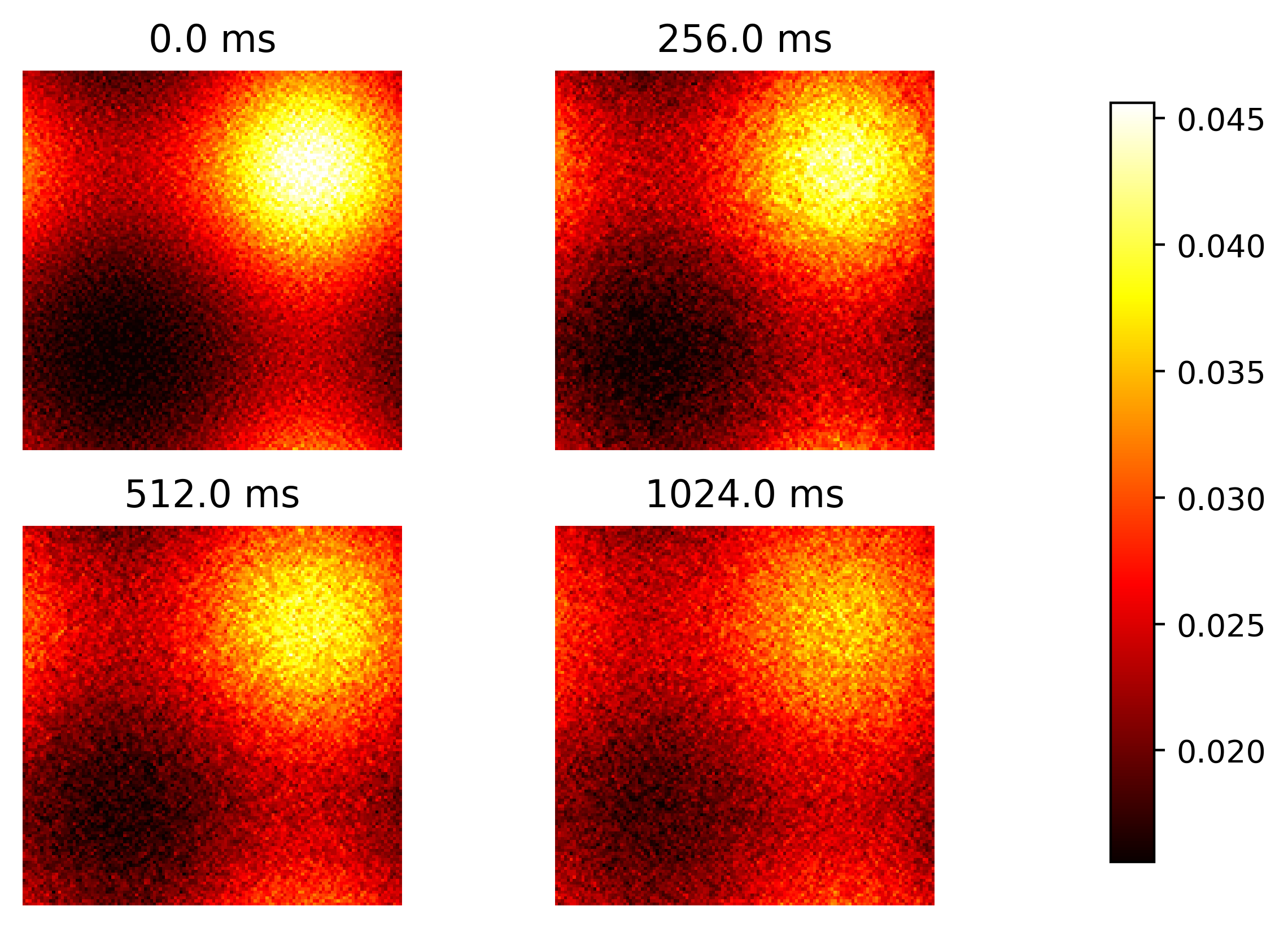}
\includegraphics[width=0.49\linewidth]{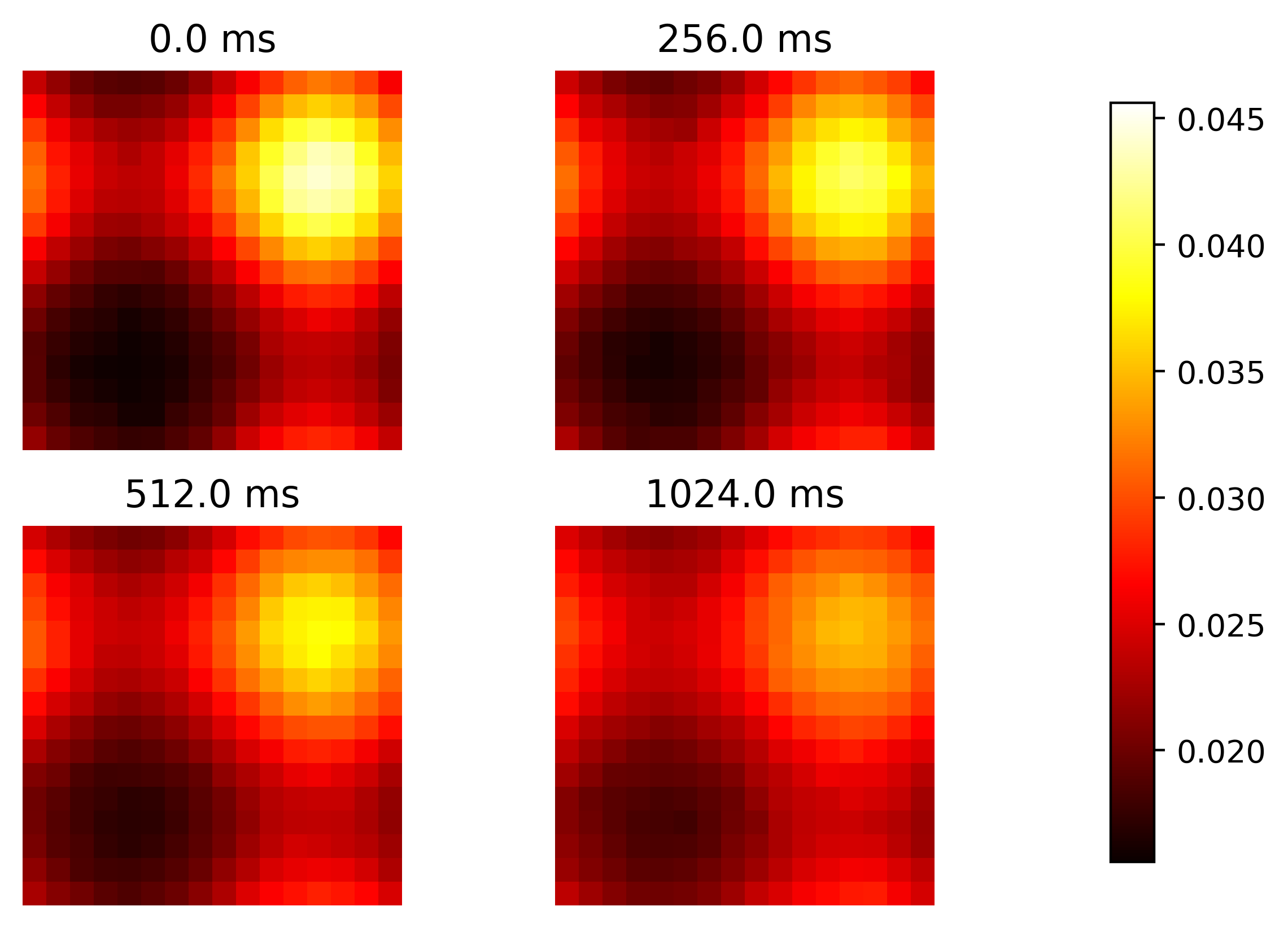}
\caption{\label{trajectory_snapshots_finest} 2d snapshots of a trajectory of the discrete Dean--Kawasaki equation \eqref{FullyDiscreteDK} started from $\overline{\rho}_{0,reg}$ and with $N=2\cdot 10^6$ particles. \emph{Left plot}: finest discretisation level $(h=2\pi\cdot 2^{-7}, \tau =0.001)$;  \emph{Right plot}: coarser discretisation level $(h=2\pi\cdot 2^{-4}, \tau =0.001 \cdot 4^3)$.}
\end{center}
\end{figure}

\item Time domain: we run simulations on the interval $[0,2^{10}\cdot t_{min}]=[0,1.024]$.
\item Differential operators: the second-order operators used in  \eqref{FullyDiscreteDK} are
\begin{align*}
\Delta_h a_h(x) & := \frac{-4a_h(x)+\sum_{y\sim x}a_h(y)}{h^2}, \\
\nabla_h \cdot b_h(x) & := \sum_{r=1}^{d}{\frac{[b_h]_r(x+hf_r) - [b_h]_r(x-hf_r)}{2h}}
\end{align*}
where $y\sim x$ indicates that $x,y$ are adjacent grid points, the brackets $[b_h]_r$ indicate the $r$-th component of the vector $b_h$, and $\{f_r\}_{r=1}^{d}$ is the canonical basis of $\mathbb{R}^d$.

\item Test functions: we use $\psi(x) := x^2$ and $\phi(x) := \sin(x)+\sin(y)$.
\end{itemize}

\begin{figure}[h]
\begin{center}
\includegraphics[width=0.45\linewidth]{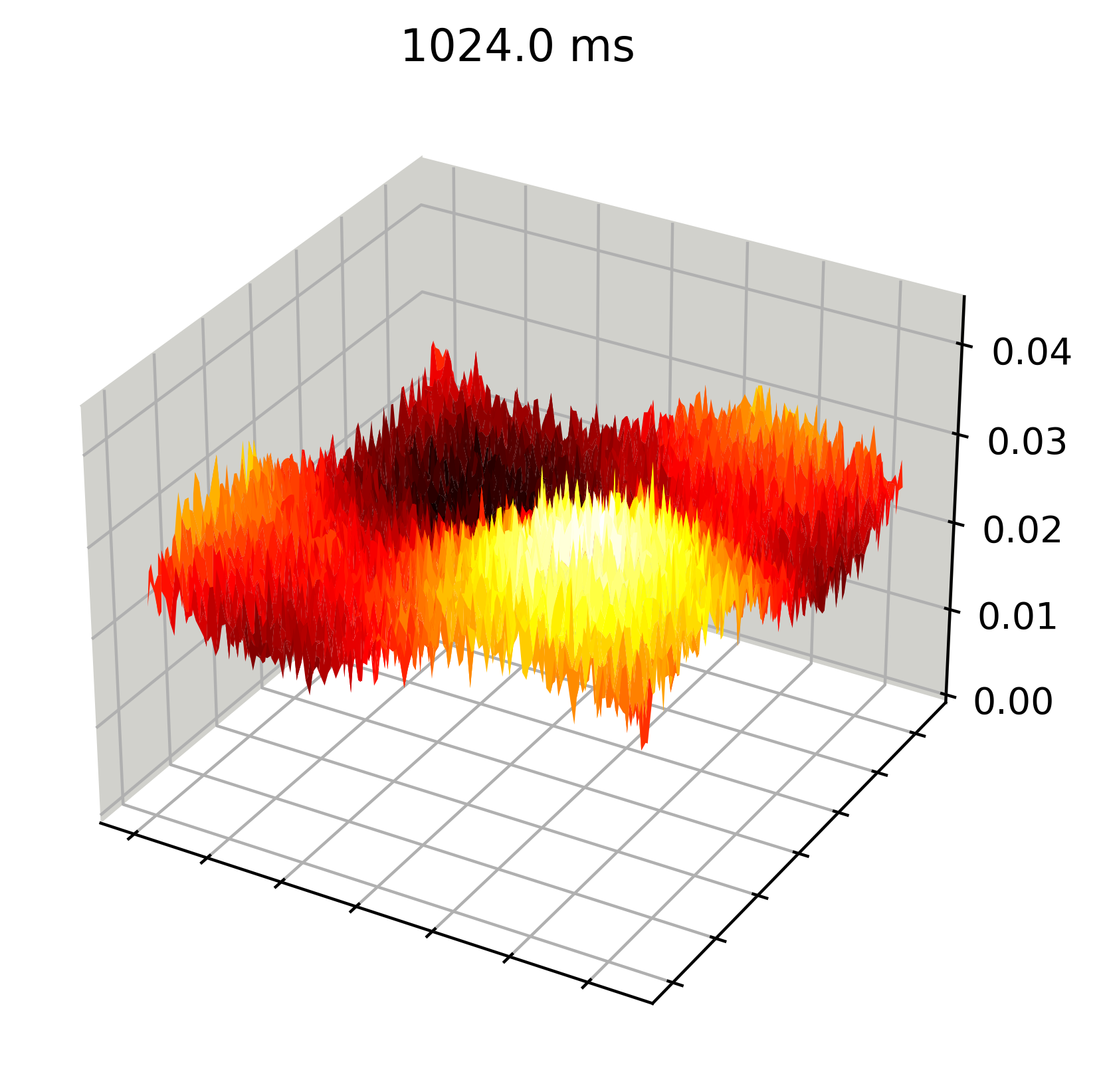}
\includegraphics[width=0.45\linewidth]{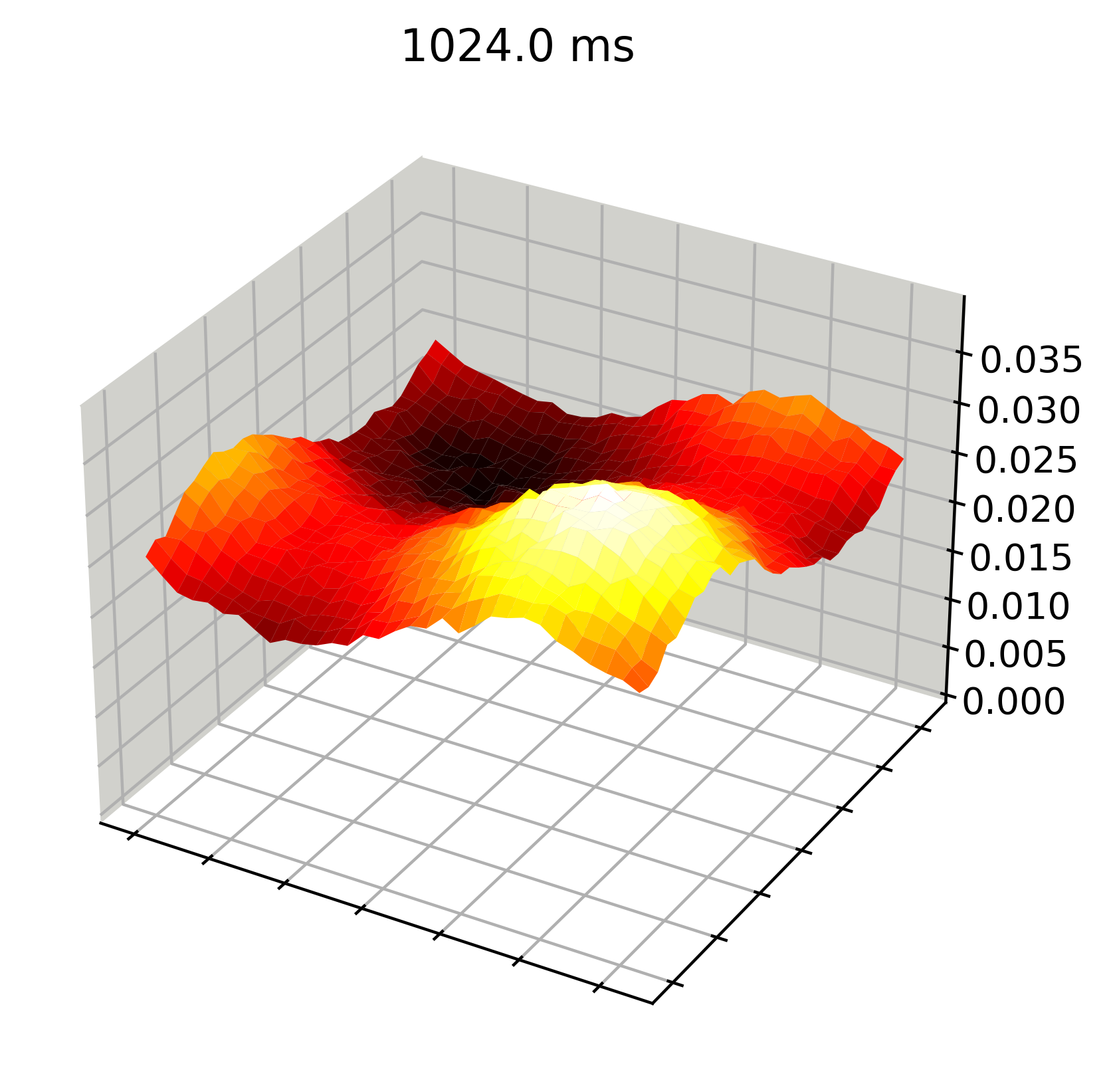}
\caption{\label{init_final_noisy_profile_rough}3d snapshots at time $T = 1024\,\, ms$ of a trajectory of the discrete Dean--Kawasaki equation \eqref{FullyDiscreteDK} started from $\overline{\rho}_{0,reg} $, and with $N=2\cdot 10^6$ particles, on different levels ($h=2\pi\cdot 2^{-7}, \tau =0.001$ for \emph{Left plot}, $h=2\pi\cdot 2^{-5}, \tau =0.001\cdot 4^2$ for \emph{Right plot}).}
\end{center}
\end{figure}

\subsection{Numerical results} We report quantitative results of the simulations for 
\begin{itemize}
\item $N=2\cdot 10^9$ and $\rho_0 = \overline{\rho}_{0,reg}$ (Figure \ref{quantitive_10_9_reg} and Table \ref{tab_9_s});
%\item $N=2\cdot 10^7$ and $\rho_0 = \overline{\rho}_{0,reg}$ (Figure \ref{quantitive_10_7_reg} and Table \ref{tab_9_s}); 
\item $N=2\cdot 10^5$ and $\rho_0 = \overline{\rho}_{0,reg}$ (Figure \ref{quantitive_10_5_reg} and Table \ref{tab_9_s});
\item $N=2\cdot 10^9$ and $\rho_0 = \overline{\rho}_{0,irreg}$ (Figure \ref{quantitive_10_9_irreg} and Table \ref{tab_9_s}).
\end{itemize}

Figures \ref{quantitive_10_9_reg}, \ref{quantitive_10_9_irreg}, \ref{quantitive_10_5_reg} are made of six plots, namely: variance estimate across levels (Top-Left), expected value estimate across levels (Top-Right), time-dependent estimate of variance for MC (Middle-Left), variance reduction factor between MC and MLMC for various varying $h_{\min}$ (Middle-Right), MLMC and MC computational cost versus accuracy $\varepsilon$ (Bottom-Left), and $\varepsilon^2$ times the MLMC and MC computational cost versus accuracy $\varepsilon$ (Bottom-Right).

We summarise our findings.

\subsubsection{Variance estimates} The quadratic decay $O(h^2_{\ell})$ for $Var[P_\ell - P_{\ell-1}]$ in \eqref{BoundVarConsecLevels} is clearly visible for the simulation with $N = 2\cdot 10^9$ and $\rho_0 = \overline{\rho}_{0,reg}$ (Figure \ref{quantitive_10_9_reg}). The same clear decay also holds in the case $N = 2\cdot 10^9$ and $\rho_0 = \overline{\rho}_{0,irreg}$ (Figure \ref{quantitive_10_9_irreg}), even though \eqref{BoundVarConsecLevels} would only predict linear decay $O(h_{\ell})$ in this case. The reason for the better-than-expected result (cfr. Remark \ref{Quadratic_vs_Linear}), is likely to be the specific nature of the mean-field limit, in which areas of low density coincide with areas of low gradient values, thus making the constant multiplying the quadratic contributions sufficiently low so as to win over linear decay.
When $N$ is decreased to $N = 2 \cdot 10^5$ (with $\rho_0 = \overline{\rho}_{0,reg}$), the quadratic decay breaks down for the lowest value of $h$ (Figure \ref{quantitive_10_5_reg}): this is to be expected from \eqref{BoundVarConsecLevels}, as the contribution $\propto N^{-1}h_\ell^{-2}$ starts dominating over the numerical error $\propto h^2_\ell$ at this point.
\subsubsection{Expected values estimates} Quadratic decay $O(h^2_{\ell})$ for $\mean{P_\ell - P_{\ell-1}}$ is clearly visible in Figures \ref{quantitive_10_9_reg}, \ref{quantitive_10_9_irreg}, where $N = 2\cdot 10^9$. When $N=2\cdot 10^5$ (Figure \ref{quantitive_10_5_reg}), the bound saturates its $h$-accuracy, and  breaks down just like the one for the variance. This is consistent with the results showed in \cite{cornalba2021dean}, and validates the systematic error of our method.

\subsubsection{Variance decay for MC}\label{sec_var_MC_decay} We plot the variance of the standard Montecarlo estimator on level $L=5$ as a function of the computational time. 
%Whenever the computational time is too large, we provide a variance extrapolation based on the last computed values. 

\subsubsection{Variance reduction of MLMC over MC}

In all cases with the exception of the simulation with $N = 2\cdot 10^5$, there is clear agreement between the variance reduction between MC/MLMC methods, and its ideal predicted growth $O(h^{-2}_{\min}|\log_2(h_{\min})|^{-1})$. 
%The variance reduction for the simulations with $N = 2\cdot 10^7$ and $N = 2\cdot 10^9$ started from $\rho_0 = \overline{\rho}_{0,reg}$ is roughly the same, as it should be since in both cases the variance decay is almost perfectly quadratic and the initial datum is the same. 
For $N=2\cdot 10^5$, the variance reduction factor stops increasing once the smallest value of $h_{\min}$ is reached. This is consistent with \eqref{GainFactor}.

\subsubsection{Accuracy $\varepsilon$ versus computational time (and $\varepsilon^2\times$computational time)} For all experiments involving the regular initial condition $\overline{\rho}_{0} = \overline{\rho}_{0,reg}$, we run Algorithm \ref{algorithm} for these values for the accuracy $\varepsilon$: 
\begin{align}\label{accuracy_val}
\varepsilon = 10^{-1.4}, 10^{-1.7}, 10^{-2}, 10^{-2.2}, 10^{-2.4}, 10^{-2.6}.
\end{align}
For the experiment involving the irregular initial condition $\overline{\rho}_{0} = \overline{\rho}_{0,irreg}$, we run Algorithm \ref{algorithm} for these values for the accuracy $\varepsilon$:
\begin{align}\label{accuracy_val_ir}
\varepsilon = 10^{-1.4}, 10^{-1.7}, 10^{-2}, 10^{-2.2}, 10^{-2.4}, 10^{-2.75}.
\end{align}
Concurrently, for each one of the accuracy values, we also run the standard MC method on the finest level reached by the MLMC method on that very accuracy value. Due to high computational cost, the running time for the MC method for the smallest values of $\varepsilon$ is extrapolated based on the MC variance trend.

We see that the MLMC provides a tangible computational gain across all simulations. This is especially visible for the smallest accuracy value across all simulations (see Table \ref{tab_9_s}). In the experiment with $\overline{\rho}_{0} = \overline{\rho}_{0,irreg}$, the MLMC method reaches level $L=5$, while in all other simulations the maximum level reached is level $L=4$ (this is reflected with a higher MLMC gain, see Table \ref{tab_9_s}). This is simply due to the systematic error decay having lower multiplicative constants for the experiments with regular datum $\overline{\rho}_{0} = \overline{\rho}_{0,reg}$. Therefore, in order to see the level change $L=5$ also for these simulations, we would have had to set a much lower accuracy (even lower that $\varepsilon = 10^{-2.75}$), and this was beyond our computing capabilities.

In the experiment with $\overline{\rho}_{0} = \overline{\rho}_{0,irreg}$, where level $L=5$ is reached, we see the beginning of the asymptotic region in $\varepsilon$, with MLMC and MC rates aligning with those given by Theorem \ref{main_res}. This is visible in Figure \ref{quantitive_10_9_irreg}, where the trend of $\varepsilon^2\times$({\tt computational time}) starts behaving like $O(\varepsilon^2 \cdot \varepsilon^{-2} \cdot \varepsilon^{-2/2}) = O(\varepsilon^{-1})$ for the MLMC algorithm, and like $O(\varepsilon^2 \cdot \varepsilon^{-2} \cdot \varepsilon^{-2/2-1}) = O(\varepsilon^{-2})$ for the standard MC method.
\begin{remark}
Additionally, note that the performance of the MLMC algorithm with $N=2\cdot 10^5$ and regular datum is similar to that of those with same datum and larger number of particles $N = 2\cdot 10^9$. This is because, for the accuracies considered, level $L=5$ -- at which the variance estimate breaks down -- is never reached.
\end{remark}

\section{Summary of work done, and final considerations}\label{summary_section}
In this paper, we have proved analytically and verified numerically that MLMC methods provide a speed-up in the computation of discretizations of the Dean-Kawasaki equation \eqref{DK}, despite this being a highly singular SPDE. We have proposed two different types of noise couplings (``Fourier coupling", Subsection \ref{secFourierCoupling}, and ``Nearest Neighbours (NN) coupling", Subsection \ref{secRM}) which grant bounds decaying like a power law in $2^\ell$ for cross-level variance and systematic error on all levels $\ell$ for which the average particle density $N\hl^d$ if sufficiently large (see \eqref{informal_lower_bound_density}). This has allowed us to get the MLMC complexity result \eqref{MLMC_complexity}, and the MLMC/MC variance reduction result \eqref{factor_var_red}.

While the simulations conducted so far provide, in our opinion, a clear indication of the effectiveness of our results (at least in the cases we have analyzed), we also believe that further testing is needed to consolidate the results. In this respect:
\begin{itemize}
\item It would be interesting to implement the Fourier coupling, and assess its competitive advantage against the NN coupling.
\item With the exception of the experiment started from $\overline{\rho}_{0} = \overline{\rho}_{0,irreg}$, it is difficult to see whether the predicted rate of growth of the MLMC cost is indeed behaving like $O(\varepsilon^{-2} \cdot \varepsilon^{-1})$, as it is given in \eqref{MLMC_complexity}.
This is mainly due to the MLMC algorithm staying of relatively low levels\footnote{This appears to be due -- for the case of regular initial datum $\rho_0 = \overline{\rho}_{0,reg}$ --  to systematic error estimates having very low leading constants for the setup considered, and thus preventing the MLMC algorithm from adding too many levels.}.
A thorough and comprehensive validation of the $\varepsilon$-dependence of the bound \eqref{MLMC_complexity} goes beyond the scope of this paper.
\end{itemize}

\begin{figure}[h]
\begin{center}
\includegraphics[width=0.48\linewidth]{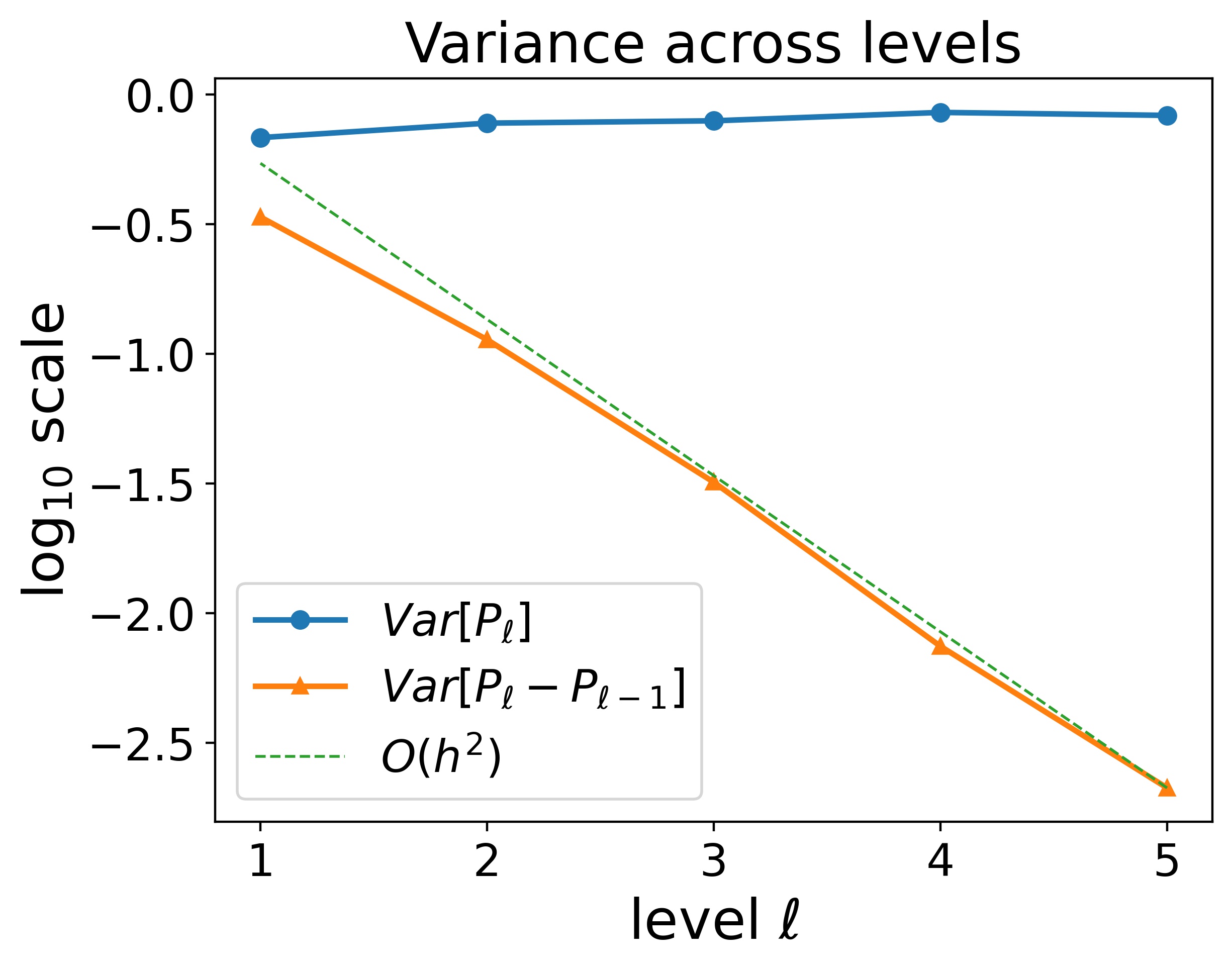}
\includegraphics[width=0.48\linewidth]{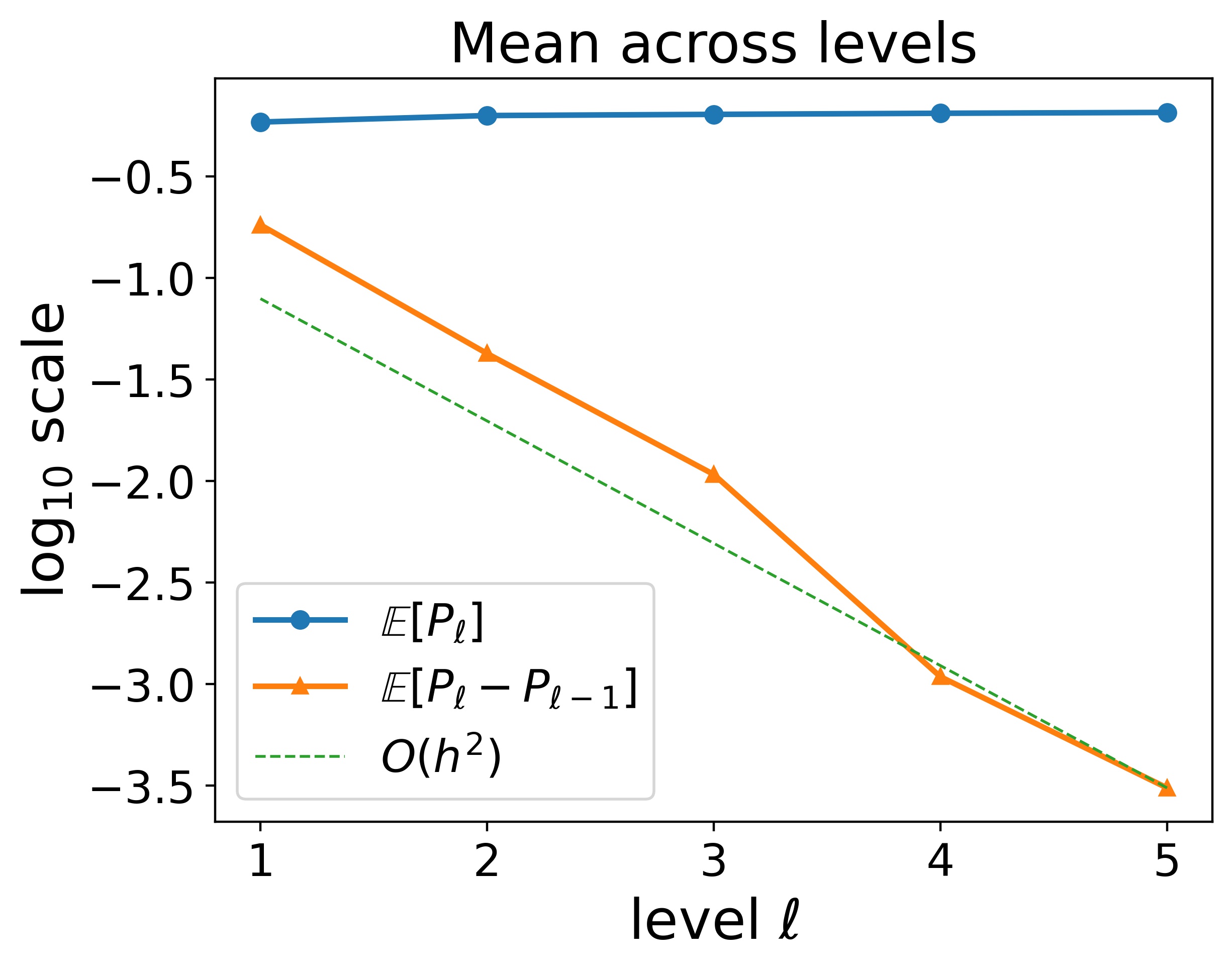}\\
\includegraphics[width=0.48\linewidth]{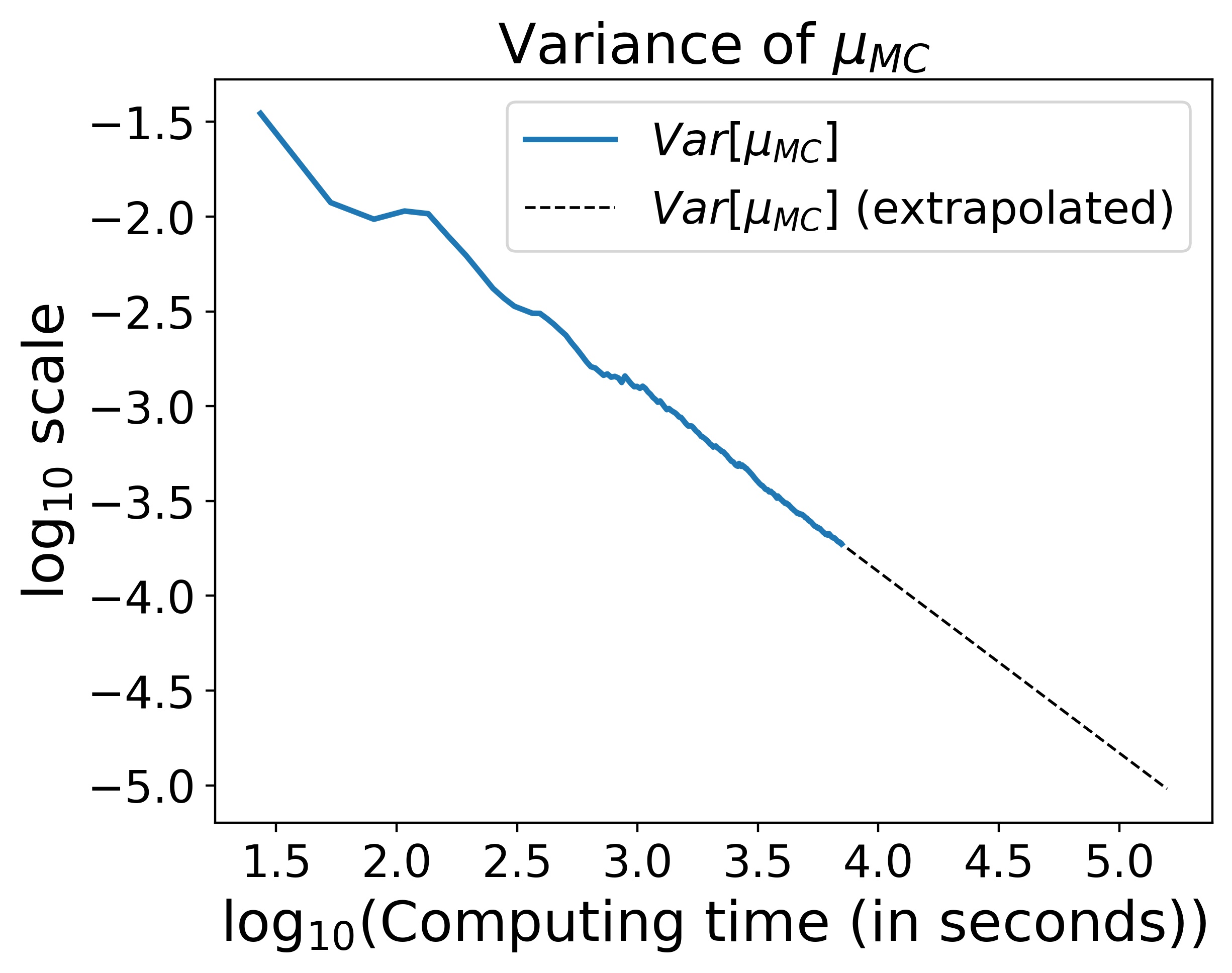}\hspace{0.7pc}
\vspace{0.7pc}
\includegraphics[width=0.48\linewidth]{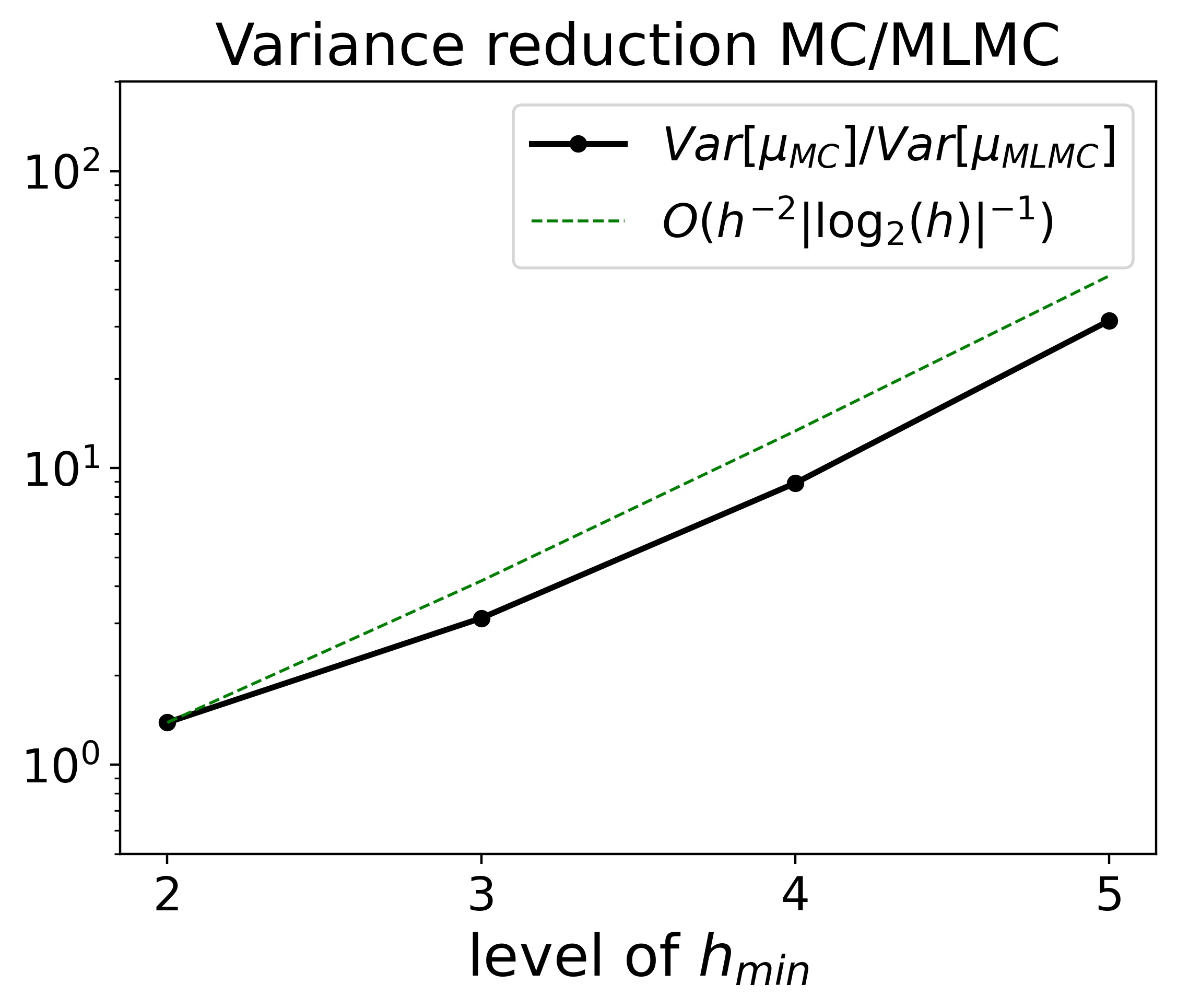}\\
\includegraphics[width=0.47\linewidth]{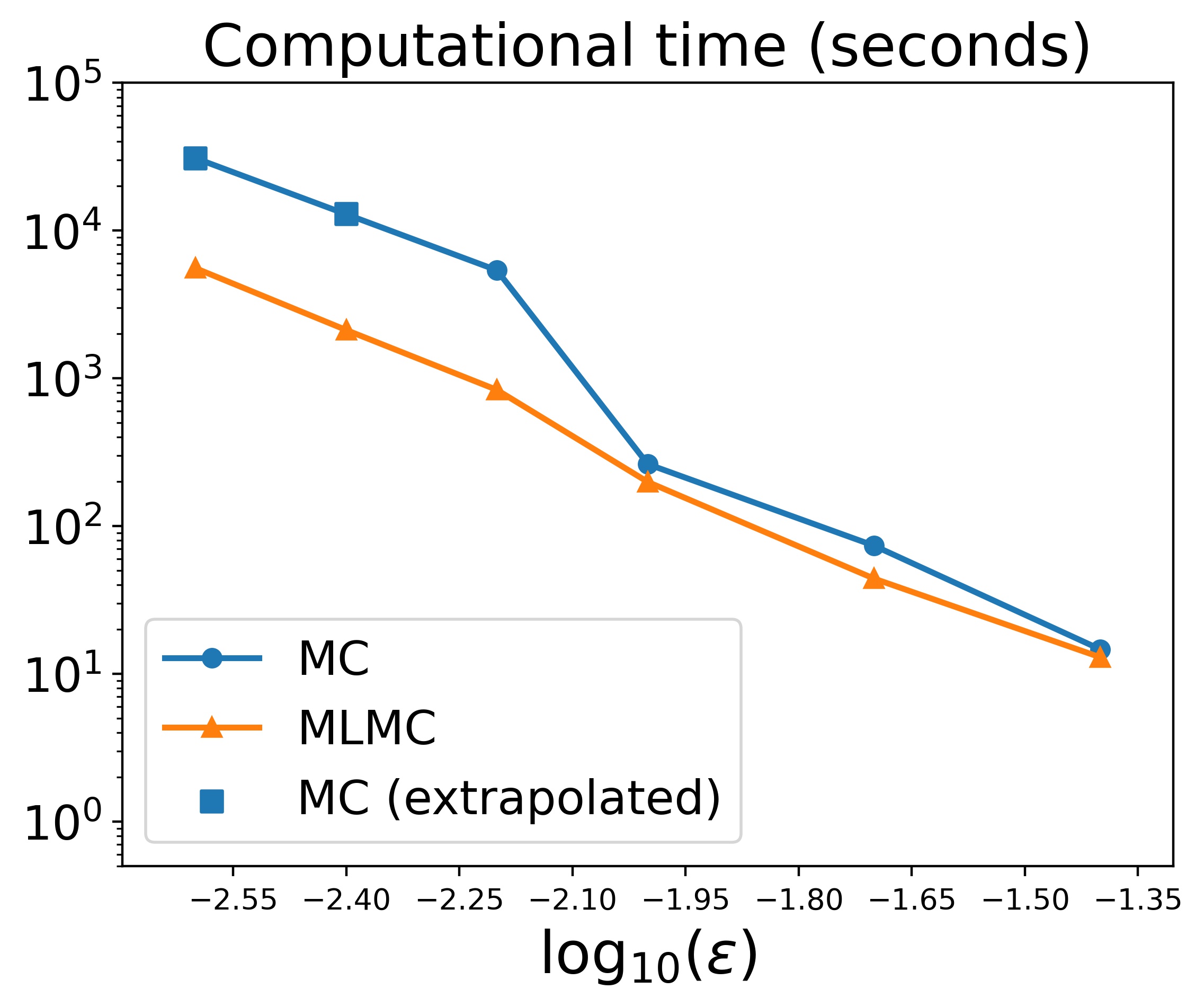}\hspace{0.7pc}
\includegraphics[width=0.49\linewidth]{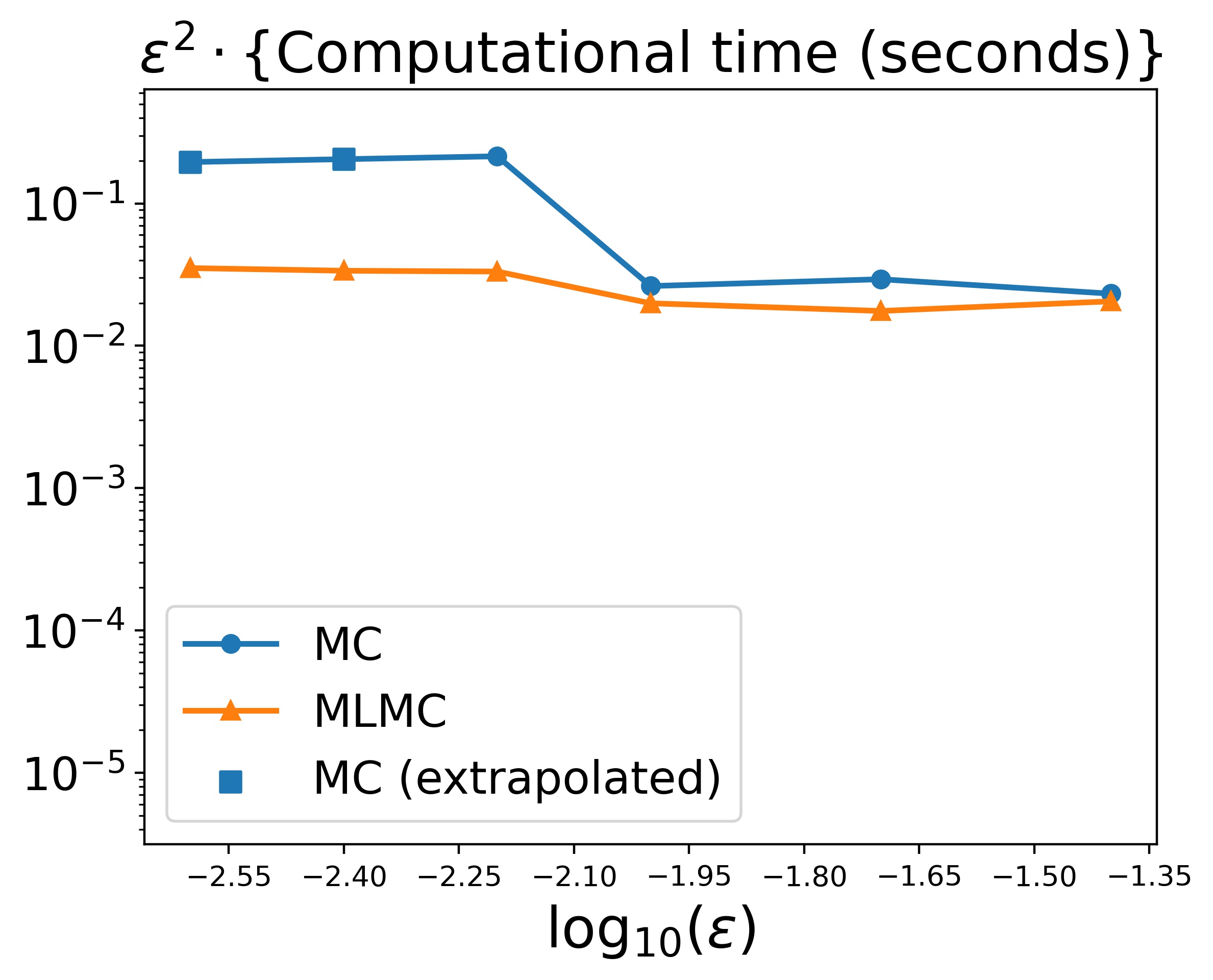}
\caption{\label{quantitive_10_9_reg} Computational results ($N=2\cdot 10^9$, $\overline{\rho}_{0} = \overline{\rho}_{0,reg}$).}
\end{center}
\end{figure}

\begin{table}[h]
     \begin{center}
     \begin{tabular}{|p{4cm}|||*{3}{c|}}
               \hline
     & Smallest $\varepsilon$  & $L$ reached & speed-up: $\frac{\mbox{time MC}}{\mbox{time MLMC}}$  \\[1ex]
               \hline\hline\hline
     $N=2\cdot 10^9$, $\overline{\rho}_{0} = \overline{\rho}_{0,reg}$   & $10^{-2.6}$ & $4$ &  $\approx 5.6$  \\
          \hline
   % $N=2\cdot 10^7$, $\overline{\rho}_{0} = \overline{\rho}_{0,reg}$ & $10^{-2.6}$ & $4$ &  $\approx 5.8$   \\
   %  \hline
    $N=2\cdot 10^5$, $\overline{\rho}_{0} = \overline{\rho}_{0,reg}$& $10^{-2.6}$ & $4$ &  $\approx 5.4$ \\
      \hline
     $N=2\cdot 10^9$, $\overline{\rho}_{0} = \overline{\rho}_{0,irreg}$ &$10^{-2.75}$ & $5$ &  $\approx 22.8$\\
      \hline
     \end{tabular}
     \end{center}
     \caption{MLMC/MC speed-up factors for smallest accuracies considered across experiments.}
     \label{tab_9_s}
\end{table}

%\begin{figure}[h]
%\begin{center}
%\includegraphics[width=0.48\linewidth]{var_20000000.jpeg}
%\includegraphics[width=0.48\linewidth]{mean_20000000.jpeg}\\
%\includegraphics[width=0.48\linewidth]{MC_MLMC_variances_20000000.jpeg}\hspace{0.7pc}
%\vspace{0.7pc}
%\includegraphics[width=0.48\linewidth]{var_red_diff_h_20000000.jpeg}\\
%\includegraphics[width=0.48\linewidth]{cost_MLMC_20000000.jpeg}\hspace{0.7pc}
%\includegraphics[width=0.48\linewidth]{eps_time_cost_MLMC_20000000.jpeg}
%\caption{\label{quantitive_10_7_reg} Computational results ($N=2\cdot 10^7$, $\overline{\rho}_{0} = \overline{\rho}_{0,reg}$).}
%\end{center}
%\end{figure}

%\begin{table}[H]
%     \begin{center}
%     \begin{tabular}{|p{2.5cm}||*{7}{l|}}
%          \hline
%     Accuracy $\varepsilon$  & $10^{-2.6}$ & $10^{-2.4}$ & $10^{-2.2}$ &  $10^{-2}$ & $10^{-1.7}$ & $10^{-1.4}$ & $10^{-1.1}$ \\
%     \hline
%     Level $L$ reached  & $4$ & $4$ & $4$ &  $3$ & $3$ & $3$ & $2$ \\
%      \hline
%     $\frac{\mbox{time MC}}{\mbox{time MLMC}}$ & $\dots$ & $\dots$ & $\dots$ &  $\dots$ & $\dots$ & $\dots$ & $\dots$ \\
%      \hline
%     \end{tabular}
%     \end{center}
%     \caption{Computational results ($N=2\cdot 10^7$, $\overline{\rho}_{0} = \overline{\rho}_{0,reg}$)}
%     \label{tab_7_s}
%\end{table}

\begin{figure}[h]
\begin{center}
\includegraphics[width=0.48\linewidth]{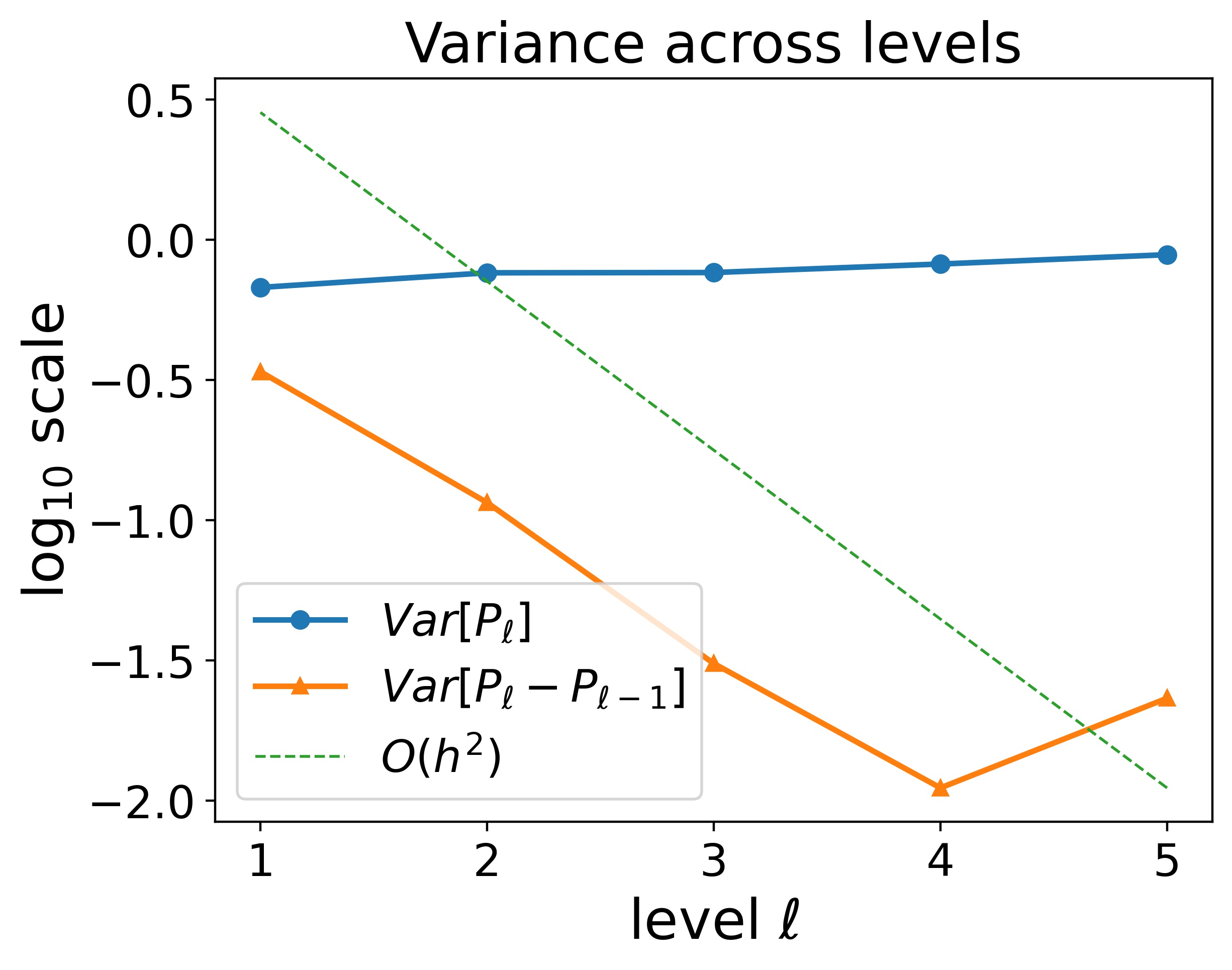}
\includegraphics[width=0.48\linewidth]{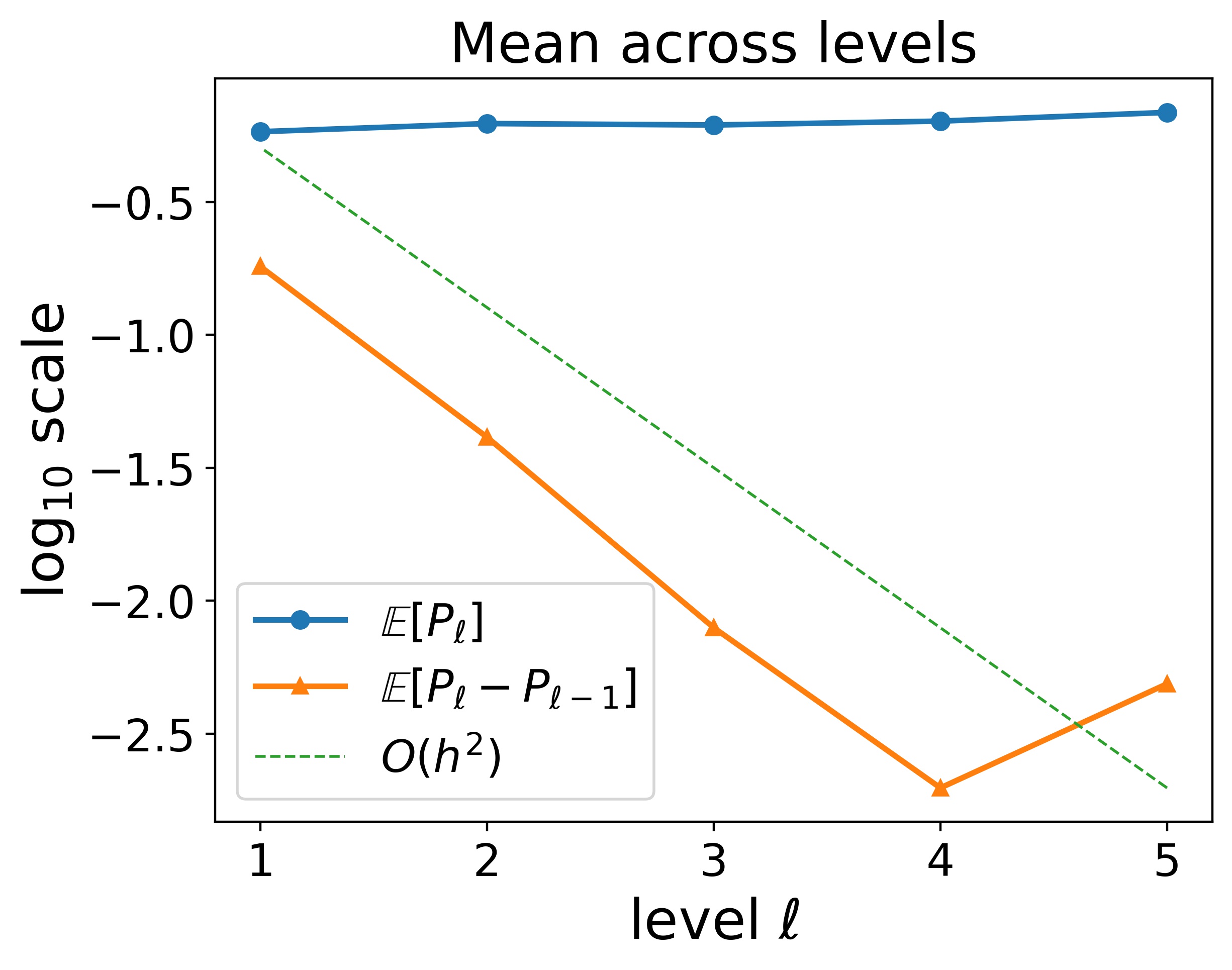}\\
\includegraphics[width=0.48\linewidth]{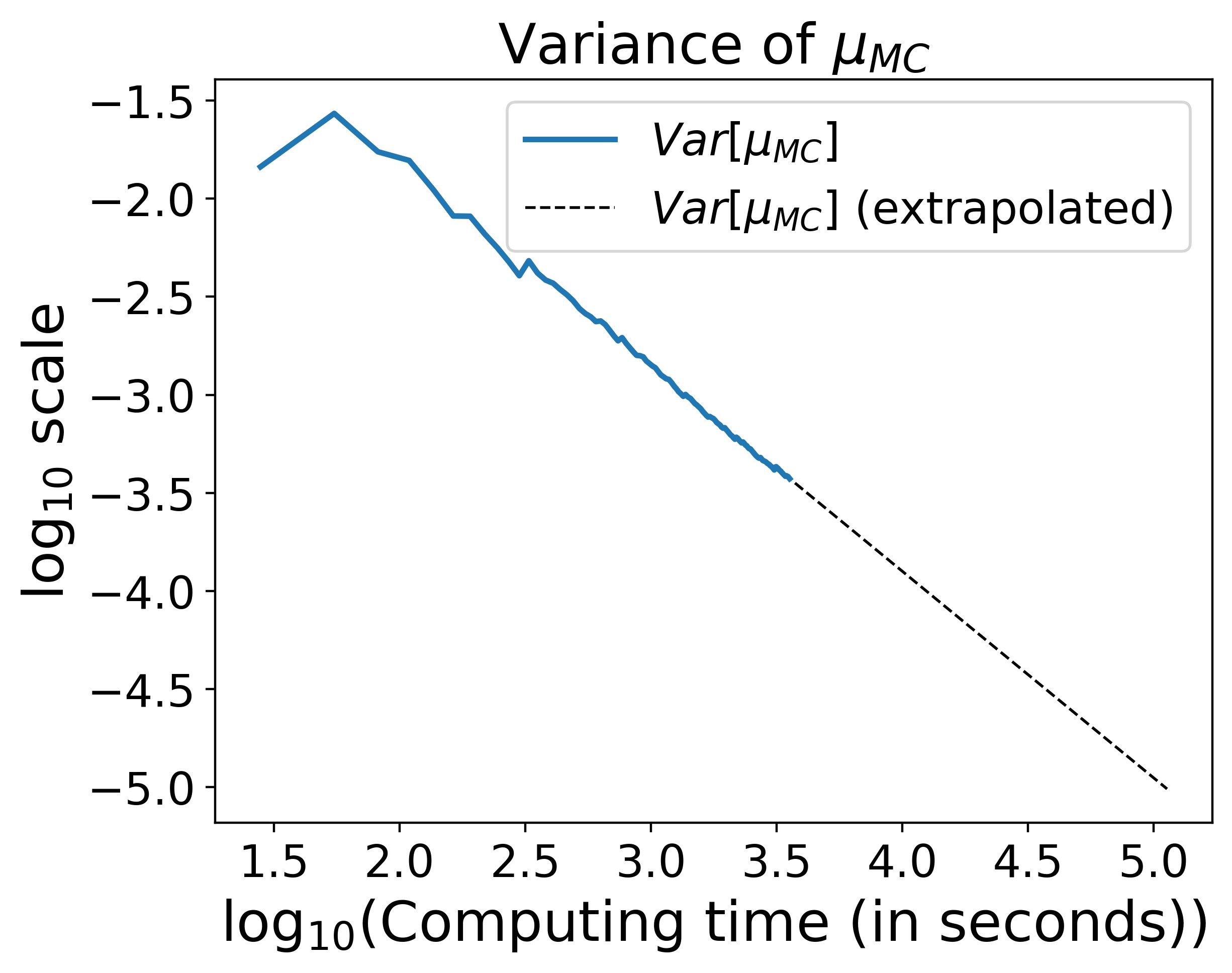}\hspace{0.7pc}
\vspace{0.7pc}
\includegraphics[width=0.48\linewidth]{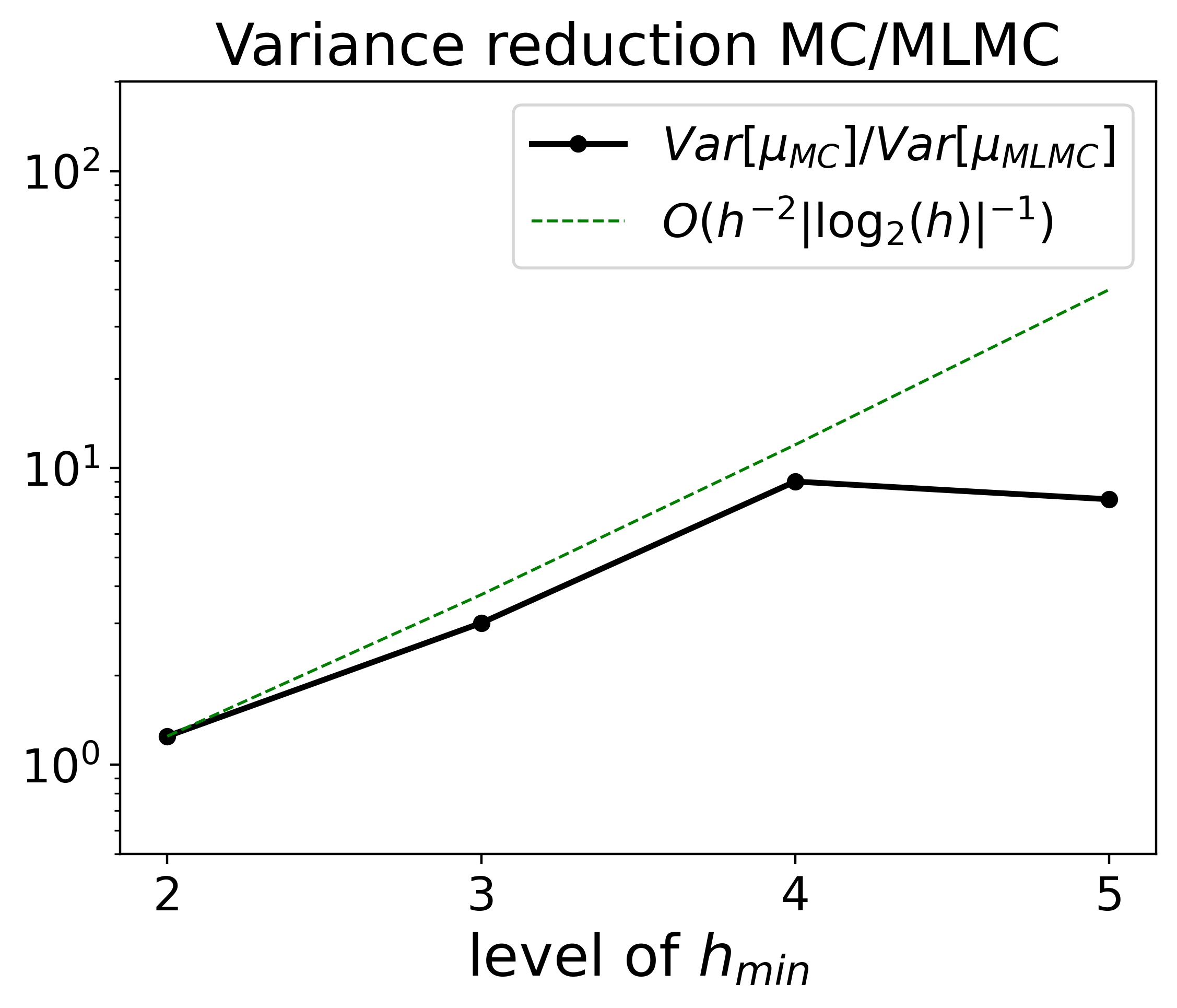}\\
\includegraphics[width=0.47\linewidth]{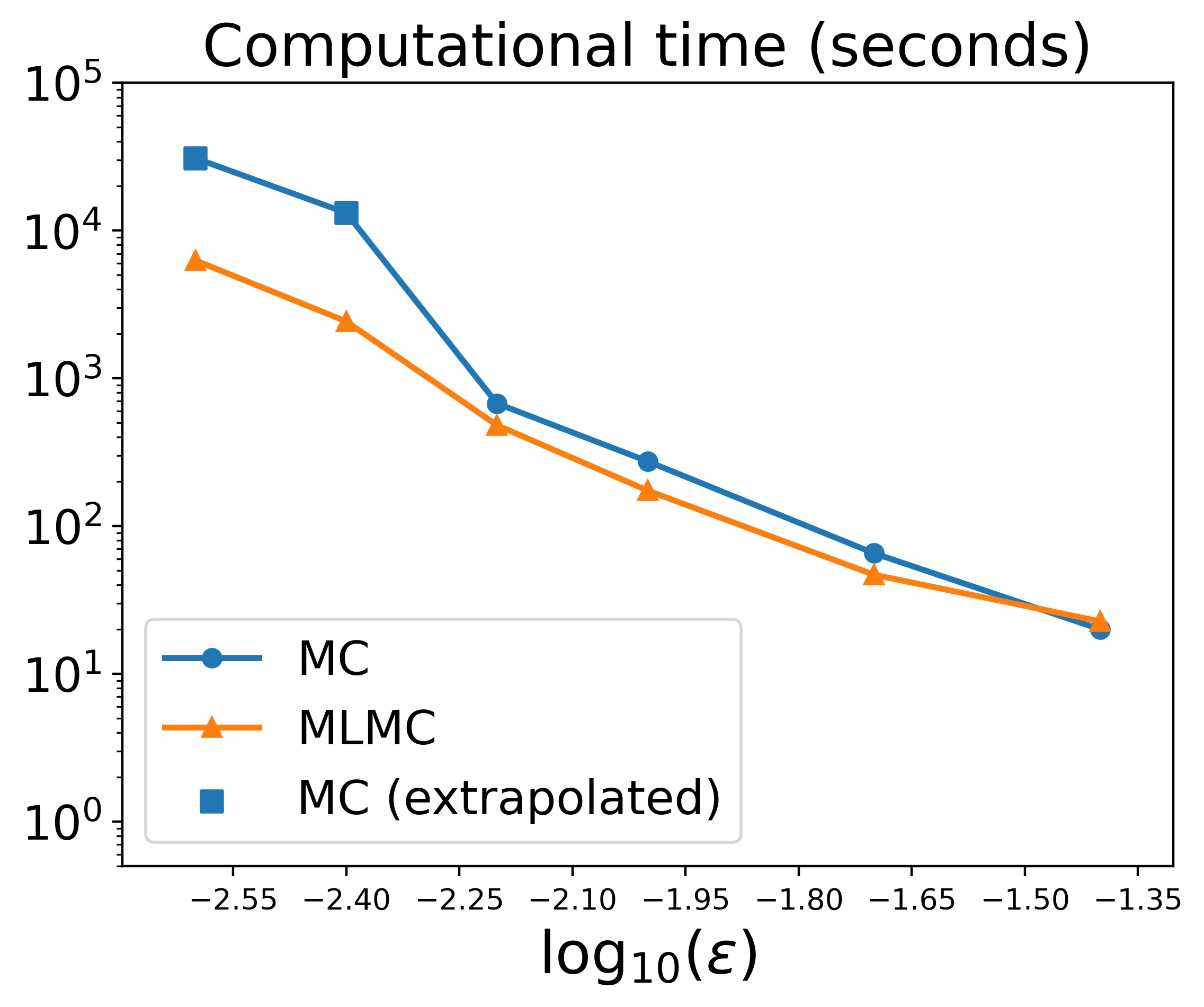}\hspace{0.7pc}
\includegraphics[width=0.49\linewidth]{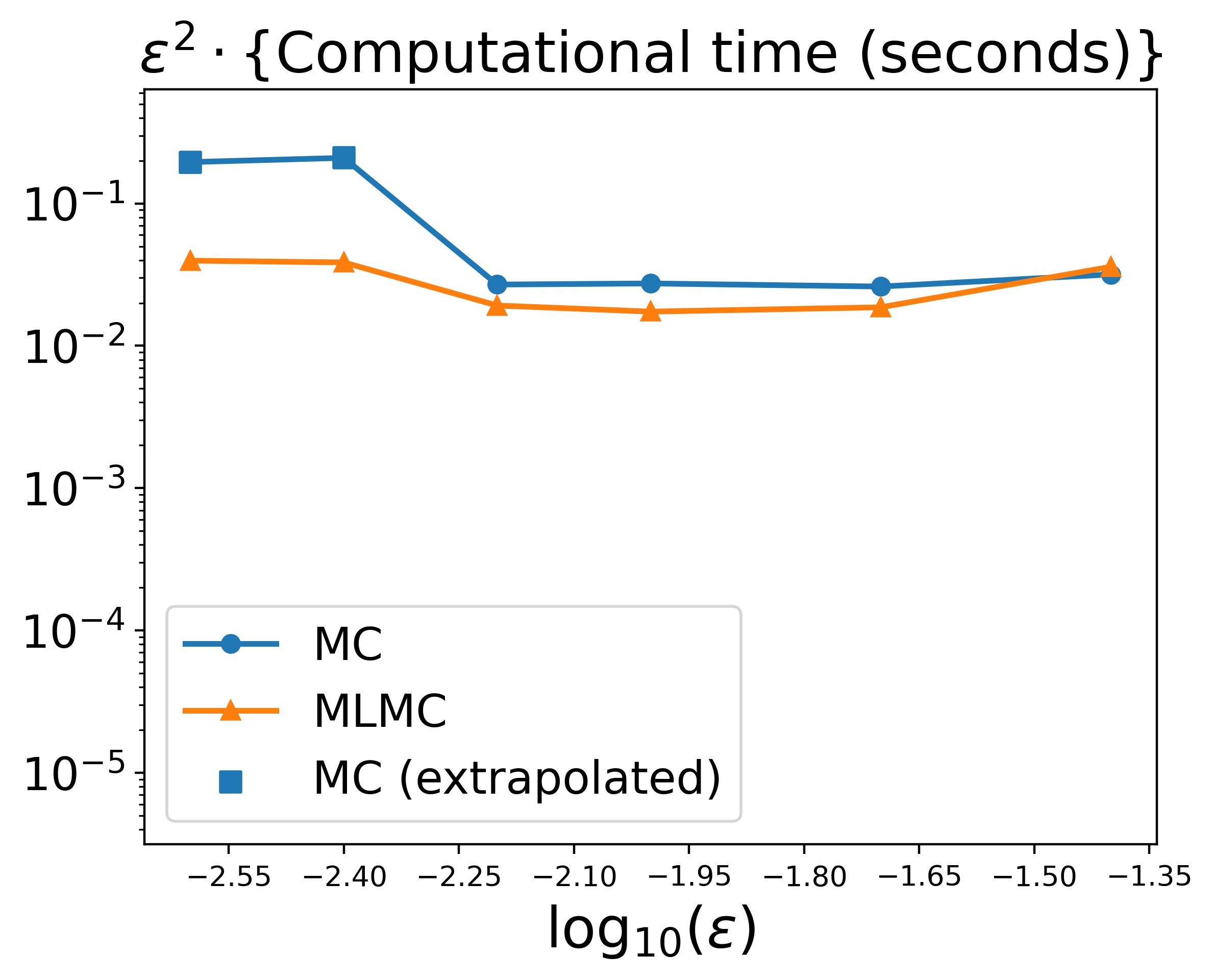}
\caption{\label{quantitive_10_5_reg} Computational results ($N=2\cdot 10^5$, $\overline{\rho}_{0} = \overline{\rho}_{0,reg}$).}
\end{center}
\end{figure}

%\begin{table}[H]
%     \begin{center}
%     \begin{tabular}{|p{2.5cm}||*{7}{l|}}
%          \hline
%     Accuracy $\varepsilon$  & $10^{-2.6}$ & $10^{-2.4}$ & $10^{-2.2}$ &  $10^{-2}$ & $10^{-1.7}$ & $10^{-1.4}$ & $10^{-1.1}$ \\
%     \hline
%     Level $L$ reached  & $4$ & $4$ & $4$ &  $3$ & $3$ & $3$ & $2$ \\
%      \hline
%     $\frac{\mbox{time MC}}{\mbox{time MLMC}}$ & $\dots$ & $\dots$ & $\dots$ &  $\dots$ & $\dots$ & $\dots$ & $\dots$ \\
%      \hline
%     \end{tabular}
%     \end{center}
%     \caption{Computational results ($N=2\cdot 10^5$, $\overline{\rho}_{0} = \overline{\rho}_{0,reg}$)}
%     \label{tab_5_s}
%\end{table}

\begin{figure}[h]
\begin{center}
\includegraphics[width=0.48\linewidth]{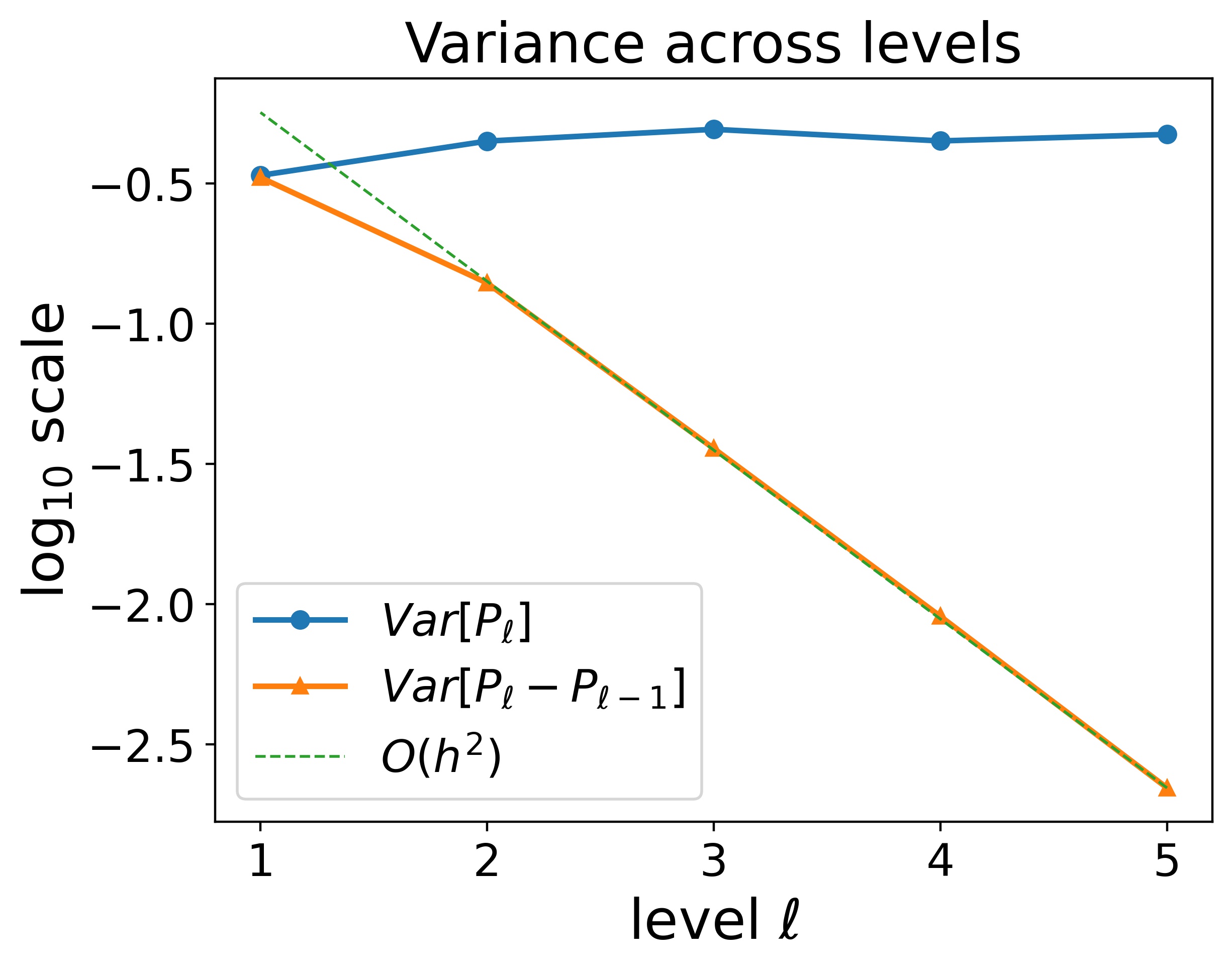}
\includegraphics[width=0.48\linewidth]{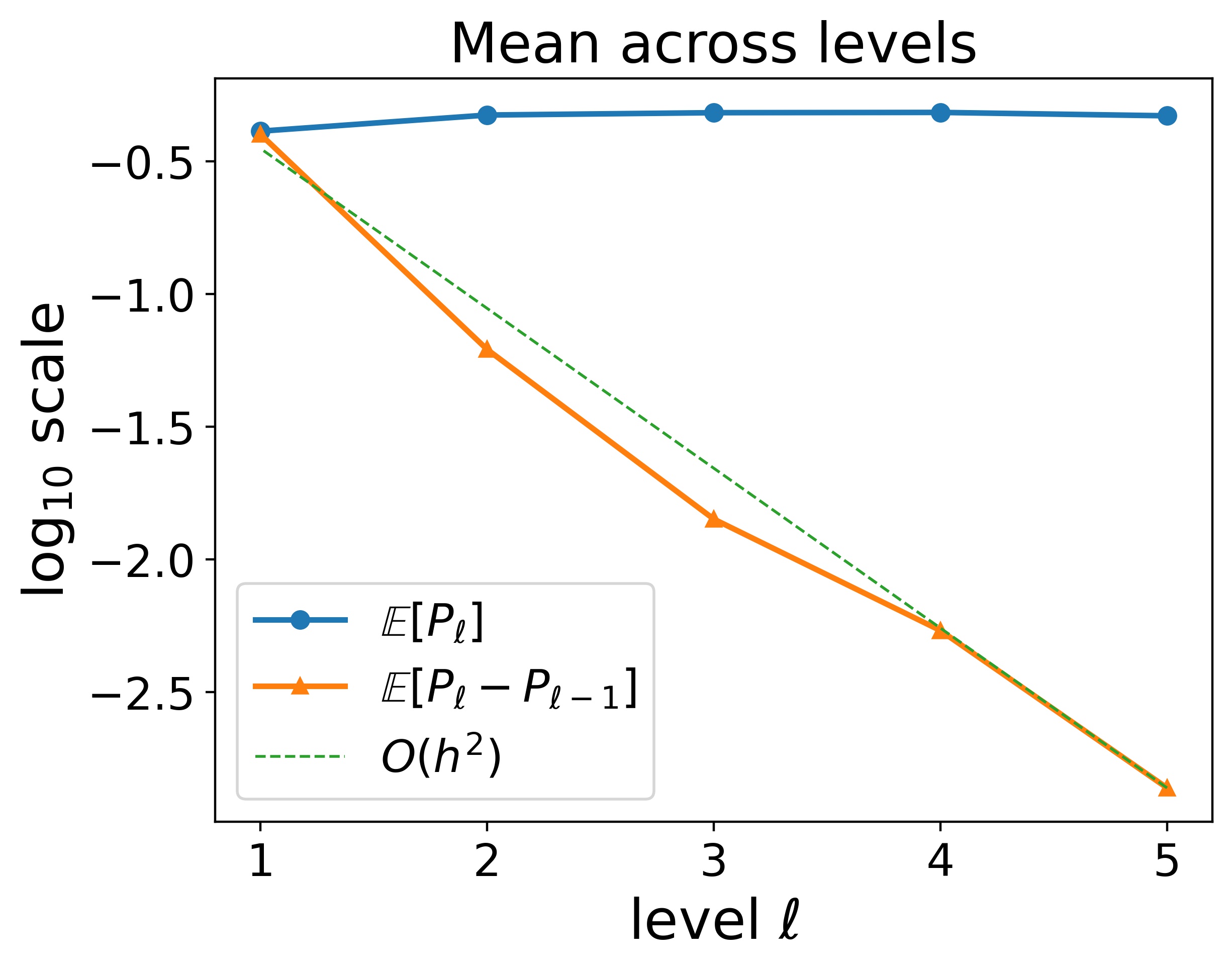}\\
\includegraphics[width=0.48\linewidth]{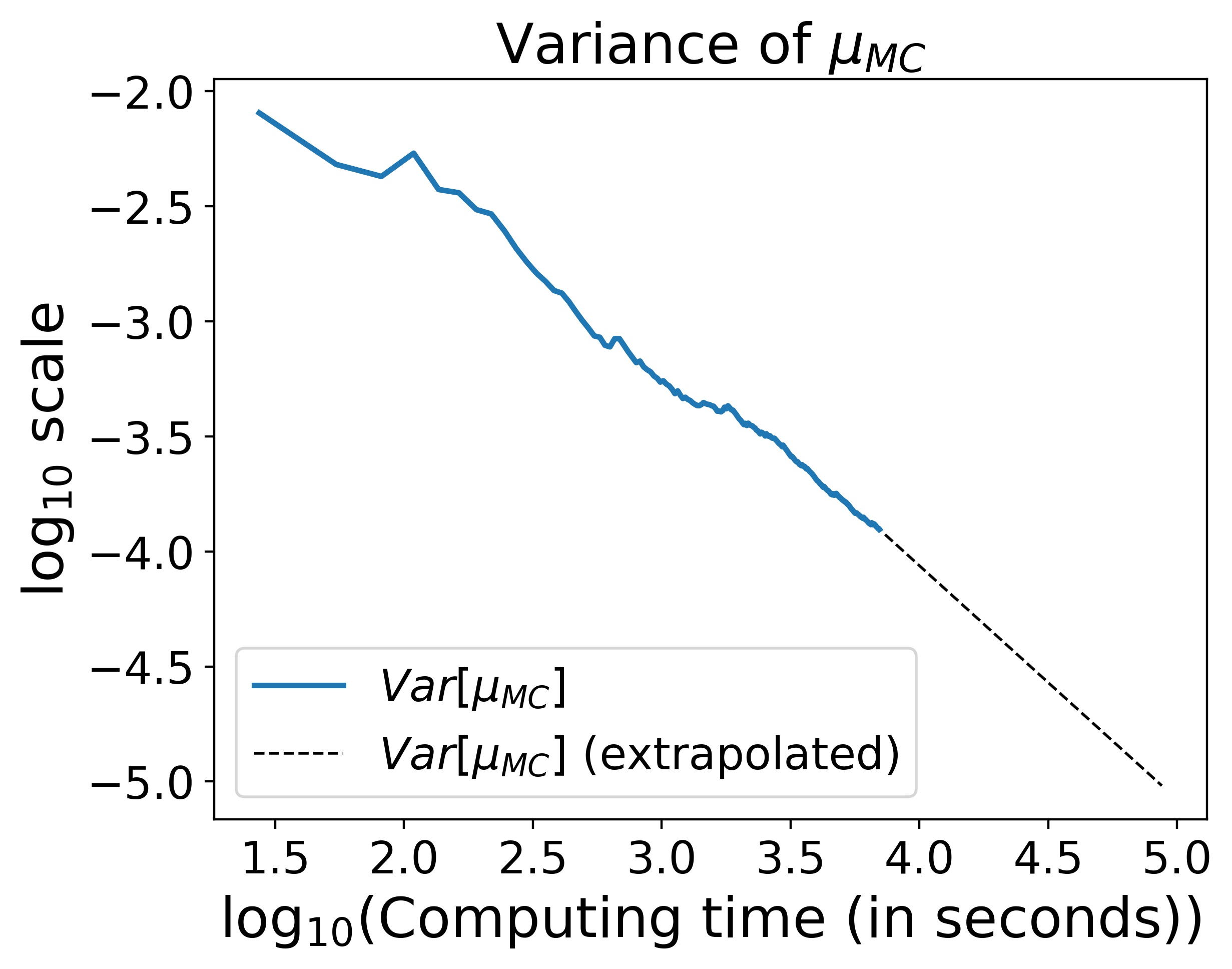}\hspace{0.7pc}
\vspace{0.7pc}
\includegraphics[width=0.48\linewidth]{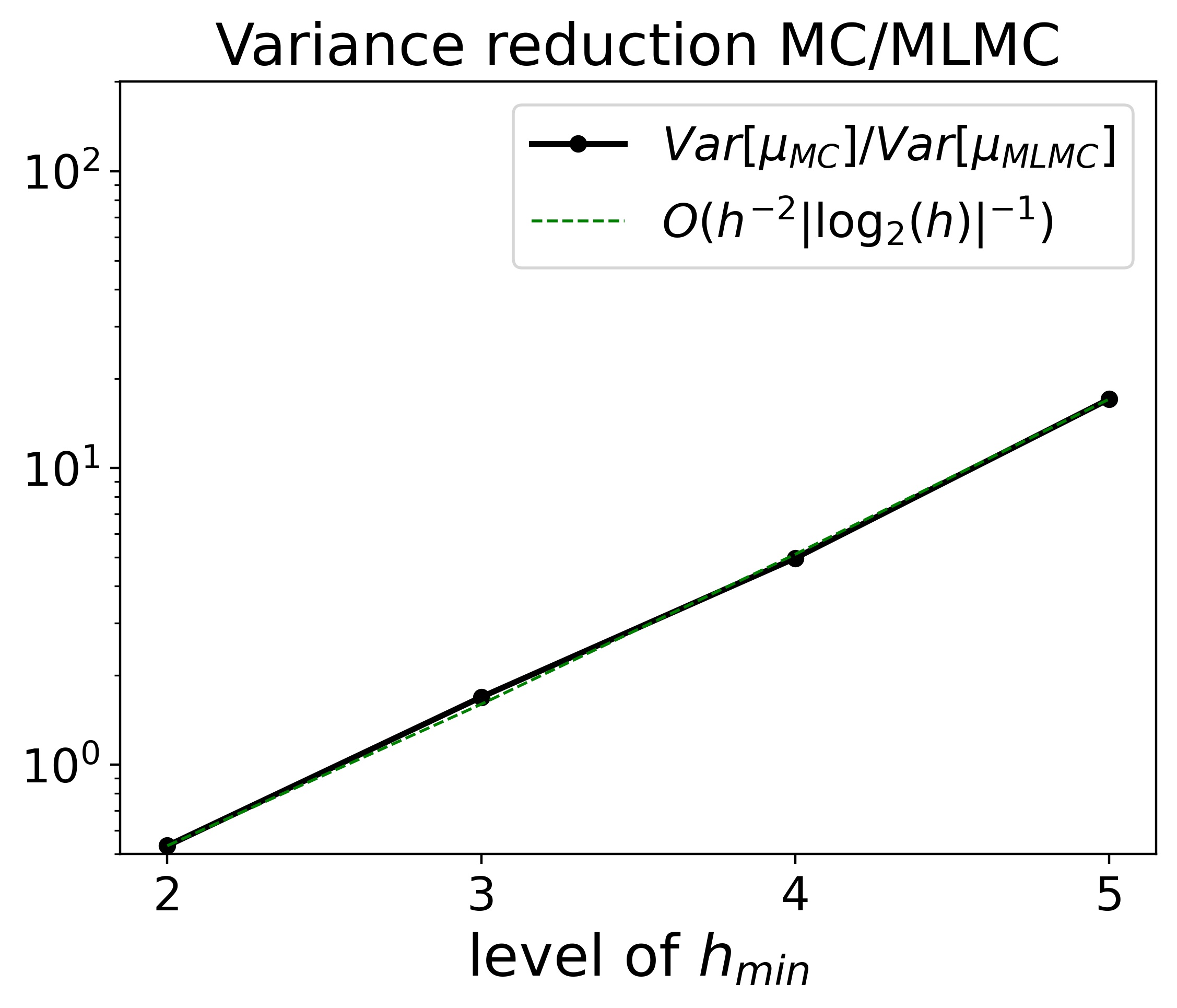}\\
\includegraphics[width=0.47\linewidth]{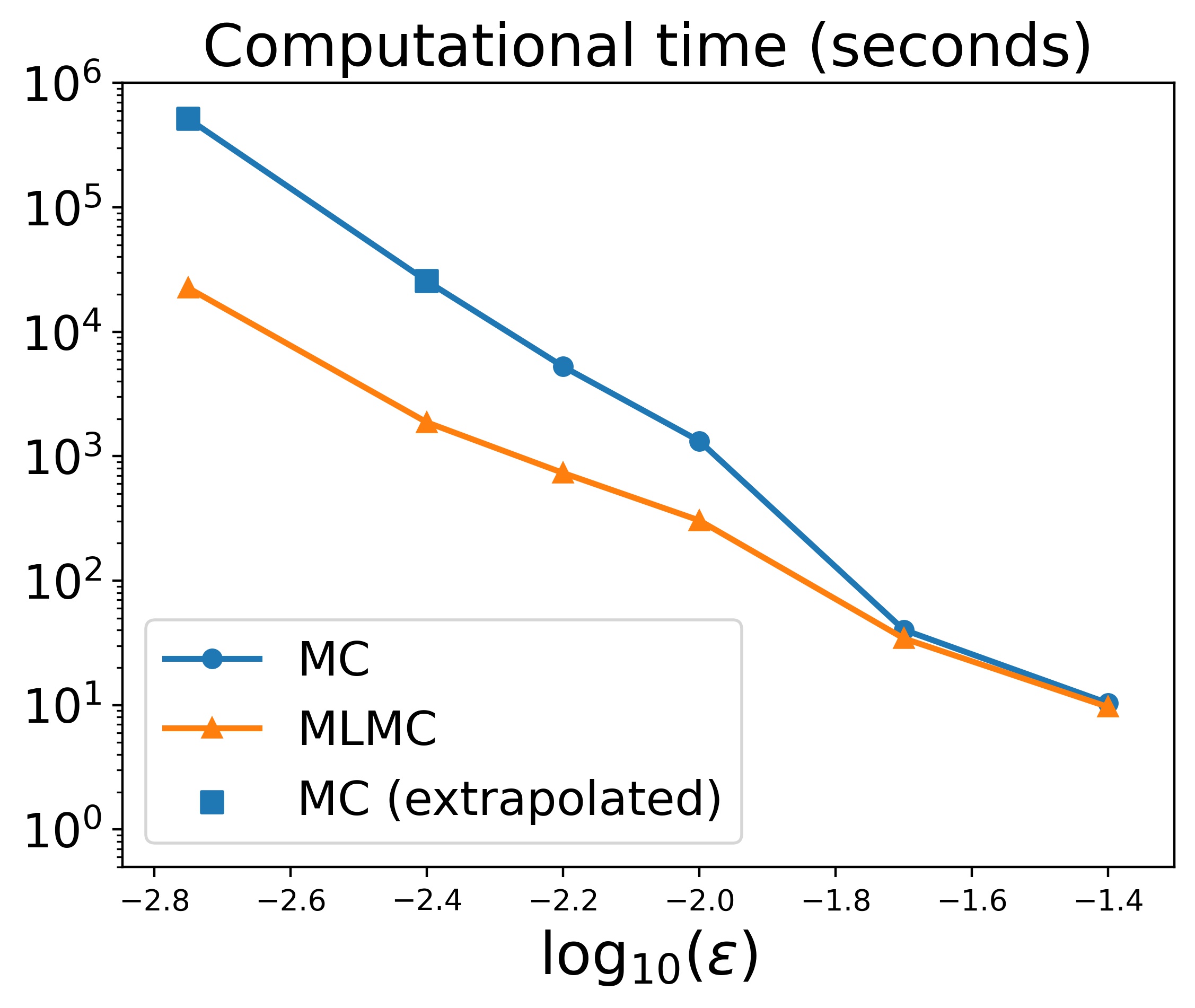}\hspace{0.7pc}
\includegraphics[width=0.49\linewidth]{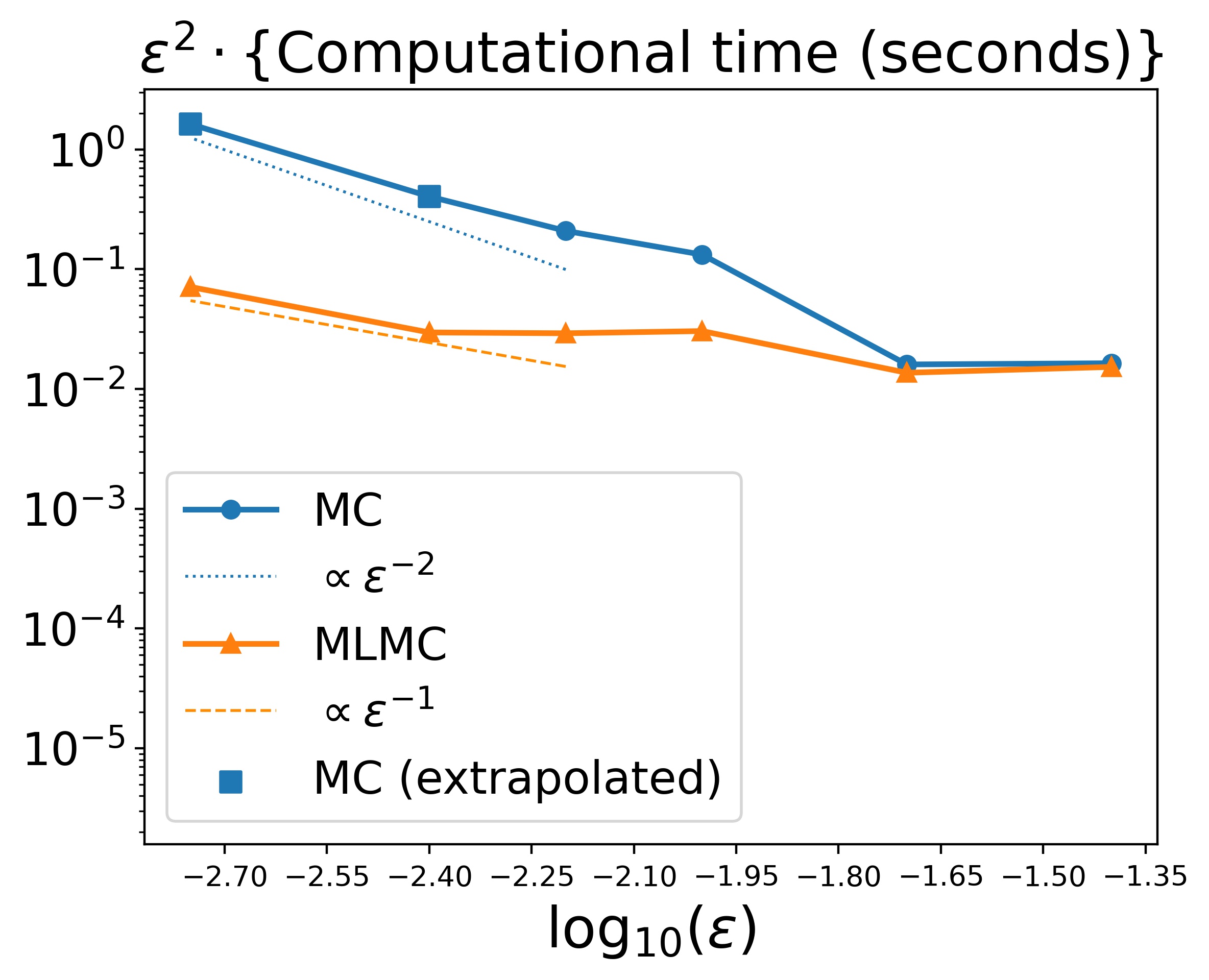}
\caption{\label{quantitive_10_9_irreg} Computational results ($N=2\cdot 10^9$, $\overline{\rho}_{0} = \overline{\rho}_{0,irreg}$).}
\end{center}
\end{figure}

%\begin{table}[H]
%     \begin{center}
%     \begin{tabular}{|p{2.5cm}||*{7}{l|}}
%          \hline
%     Accuracy $\varepsilon$  & $10^{-2.6}$ & $10^{-2.4}$ & $10^{-2.2}$ &  $10^{-2}$ & $10^{-1.7}$ & $10^{-1.4}$ & $10^{-1.1}$ \\
%     \hline
%     Level $L$ reached  & $4$ & $4$ & $4$ &  $3$ & $3$ & $3$ & $2$ \\
%      \hline
%     $\frac{\mbox{time MC}}{\mbox{time MLMC}}$ & $\dots$ & $\dots$ & $\dots$ &  $\dots$ & $\dots$ & $\dots$ & $\dots$ \\
%      \hline
%     \end{tabular}
%     \end{center}
%     \caption{Computational results ($N=2\cdot 10^9$, $\overline{\rho}_{0} = \overline{\rho}_{0,irreg}$)}
%     \label{tab_9_r}
%\end{table}

\appendix

\section{Estimates for fully discrete setting}\label{AppendixC}
In what follows, it is convenient to use the short-hand notation 
$
\cfl := \tau/h^2
$ for the standard CFL quotient.

\begin{lemma}[Discrete energy estimates]\label{Discrete_Energy_Est}
We define the test function dynamics
\begin{align}\label{BackwardsTestOneStep}
\phi_{h,\Dt}^{m\Dt} = A_h \phi_{h,\Dt}^{(m+1)\Dt},
\end{align}
with 
$
A_h := (I_d- \Dt b_0\frac{1}{2}\Delta_h)^{-1}(I_d +\Dt b_1 \frac{1}{2}\Delta_h)
$, and where $b_0,b_1$ are the coefficients from \eqref{FullyDiscreteDK}, see also Assumption \ref{ass_stability}. If $b_1>b_0$, 
the energy estimate for $\phi_{h,\Dt}$
\begin{align}\label{DiscreteEnergyEstimate}
\Dt \sum_{m}{\|\nabla_h\phi^{m\Dt}_{h,\tau}\|_h^2} \lesssim (\cfl^2 \wedge 1) \|\phi^{T}_{h,\Dt}\|^2_h
\end{align}
holds subject to $\cfl$ being sufficiently small. If $b_1\leq b_0$, then \eqref{DiscreteEnergyEstimate} holds with no smallness requirement on $\cfl$.
\end{lemma}
\begin{proof}
We multiply both sides in \eqref{BackwardsTestOneStep} by $(I_d-\Dt \frac{b_0}{2}\Delta_h)$, compute the $\|\cdot\|^2_h$ norm of both resulting sides, reshuffle the terms, use the integration-by-parts formula $(f_h,\Delta_h f_h)_h = -\|\nabla_h f_h\|^2$, $\forall f_h \in L^2(\Ghd)$, and obtain
\begin{align*}
& b_0\Dt \|\nabla_h \phi_{h,\Dt}^{m\Dt}\|^2_h + b_1\Dt \|\nabla_h \phi_{h,\Dt}^{(m+1)\Dt}\|^2_h \nonumber \\
& \quad = \|\phi_{h,\Dt}^{(m+1)\Dt}\|^2_h - \|\phi_{h,\Dt}^{m\Dt}\|^2_h + \Dt^2\frac{b_1^2}{4}\|\Delta_h \phi_{h,\Dt}^{(m+1)\Dt}\|^2_h - \Dt^2\frac{b_0^2}{4}\|\Delta_h \phi_{h,\Dt}^{m\Dt}\|^2_h.
\end{align*}
Summing over $m$, we control the sum of the contributions $\Dt^2(b_1^2/4)\|\Delta_h \phi_{h,\Dt}^{(m+1)\Dt}\|^2_h - \Dt^2(b_0^2/4)\|\Delta_h \phi_{h,\Dt}^{m\Dt}\|^2_h$ by using a telescopic argument if $b_1\leq b_0$, or by absorbing the term $\Dt^2(b_1^2/4)\|\Delta_h \phi_{h,\Dt}^{(m+1)\Dt}\|^2_h$ in the term $b_1\Dt \|\nabla_h \phi_{h,\Dt}^{(m+1)\Dt}\|^2_h$ (subject to $\cfl$ being sufficiently small) if instead $b_1> b_0$. 
In all cases, \eqref{DiscreteEnergyEstimate} follows by additionally using that 
$
b_1^2\Dt^2\|\Delta_h \phi_{h,\Dt}^{T}\|_h \lesssim \cfl^2 \|\phi_{h,\Dt}^{T}\|_h^2.
$
\end{proof}
\begin{remark}\label{possible_suboptimal_discrete_energy_estimate}
As this is not crucial in this paper, we do seek to obtain the optimal dependency on $\cfl$ in Lemma \ref{Discrete_Energy_Est}.
\end{remark}

\begin{lemma}[Error estimates for standard heat equation, cfr. \cite{morton2005numerical}]\label{Lemma_ErrorDiscrTestFunc} Let $\phi$ satisfy the continuous backwards heat equation ending in $\phi^T = \varphi \in C^2$, and let $\phi_{h,\Dt}$ satisfy the discrete backwards heat equation \eqref{BackwardsTestOneStep} ending in $\Ih \varphi$. Then, for all $m$, we have 
\begin{align}\label{ErrorDiscrTestFunc}
\|\phi(\cdot,m\tau) - \phi^{m\Dt}_{h,\Dt}\|_\infty \lesssim h^2 + \Dt. 
\end{align}
\end{lemma}

\begin{lemma} For any discrete function $\phi_{h} \in H^2_{h}$, we have 
\begin{align}\label{Matrix_error}
\left\| |\nabla_h \left[ (I_d - b_0\Dt \Delta_h)^{-1}\phi_h - \phi_h \right] \right\|_h^2 \lesssim b_0^2 \Dt^2 \|\phi_h\|^2_{H^2_{h}},
\end{align}
where $H^2_h$ is the discrete Sobolev space of order two.
\end{lemma}
\begin{proof}
This follows by taking the Fourier transform of $y_h := (I_d - b_0\Dt \Delta_h)^{-1}\phi_h$, which gives 
$
\hat{y}_h(\xi) - \hat{\phi}_h(\xi) \propto (\Dt P(\xi)\hat{\phi}_h(\xi))/(1+\Dt P(\xi))
$
for some $P(\xi)\propto |\xi|^2$.
\end{proof}

\begin{lemma}[Martingale property for discrete fluctuations]\label{lem_Discrete_Martingale}
Let \eqref{BackwardsTestOneStep} be the dynamics for the test function $\phi_{h,\Dt}$.
Let $\rho_{h,\Dt}$ be the solution to \eqref{FullyDiscreteDK}. 
Then $(\rho_{h,\Dt},\phi_{h,\Dt})_h$ is a discrete time martingale.
\end{lemma}
\begin{proof}
The definition of $A_h$ entails that 
$
\rho_{h,\Dt}^{(m+1)\Dt} = A_h \rho_{h,\Dt}^{m\Dt} + \mathcal{M}_m,
$
where $\mathcal{M}_m$ encodes the noise. 
Using the symmetry of $A_h$ and \eqref{BackwardsTestOneStep}, we end the proof by writing
\begin{align*}
(\rho_{h,\Dt}^{(m+1)\Dt}, \phi_{h,\Dt}^{(m+1)\Dt})_h & = (A_h \rho_{h,\Dt}^{m\Dt} + \mathcal{M}_m, \phi_{h,\Dt}^{(m+1)\Dt})_h \\
& = (\rho_{h,\Dt}^{m\Dt} , A_h\phi_{h,\Dt}^{(m+1)\Dt})_h + (\mathcal{M}_m, \phi_{h,\Dt}^{(m+1)\Dt})_h \\
& = (\rho_{h,\Dt}^{m\Dt} , \phi_{h,\Dt}^{m\Dt})_h + (\mathcal{M}_m, \phi_{h,\Dt}^{(m+1)\Dt})_h. \qed
\end{align*}
\end{proof}
\section{Proof of Lemma \ref{lma_ErrToMFL}}\label{proof_lma_ErrToMFL}
We need the following simple estimate.

\begin{lemma}\label{lma_ExpIntBd}
For any $\beta > 0$, there exists a constant $C>0$ such that for any $\alpha \geq 1$ we have the inequality
$
\int_{0}^{\infty}{\exp\{-\alpha x^{1/\beta}/\sqrt{1+x^{1/\beta}}\}\emph{d} x} \leq C\alpha^{-\beta}. 
$\end{lemma}
\begin{proof}
Using the inequality $1+x^{1/\beta} \leq x^{1/\beta}(1+\alpha)$ (valid for $x\geq \alpha^{-\beta}$)  
we write
\begin{align}
& \int_{0}^{\alpha^{-\beta}}{\exp\left\{-\alpha\frac{x^{1/\beta}}{\sqrt{1+x^{1/\beta}}}\right\}\m x} 
+ \int_{\alpha^{-\beta}}^{\infty}{\exp\left\{-\alpha\frac{x^{1/\beta}}{\sqrt{1+x^{1/\beta}}}\right\}\m x} \nonumber\\
& \quad \lesssim_\beta\alpha^{-\beta} + \int_{0}^{\infty}{\exp\left\{-\frac{\sqrt{\alpha} y^{2}}{\sqrt{2}(4\beta)^2} \right\}y^{4\beta-1}\m y} \lesssim_\beta\alpha^{-\beta}\nonumber,
\end{align}
where we used the bound $\alpha/\sqrt{1+\alpha}\geq \sqrt{\alpha}/\sqrt{2}$ (which is valid for $\alpha\geq 1$), the change of variables $y = 4\beta x^{1/(4\beta)}$ in the second-to-last inequality, and Gaussian moment bounds in the last inequality.
\end{proof}
\begin{proof}[Proof of Lemma \ref{lma_ErrToMFL}]
Let $x \in \Grid{h}{d}$ and $t\ \in \Tgrid[\Dt]$ be fixed. Let $\varphi_h$ be the discrete Dirac delta in $x$, namely $\delta_h := \mathbf{1}_{y = x}h^{-d}$. Moreover, let $\phi_{h,\Dt}$ be given by the discrete backwards heat equation \eqref{BackwardsTestOneStep}, and with final datum $\phi_{h,\Dt}^{t} = \delta_h$.
We define the stopping time 
$
\Ts := \min\{s\in \Tgrid[\Dt]\colon {\|\rho^s_{h,\Dt} - \ov{\rho}_{h,\Dt}^s\|_{\infty} \geq B\rho_{\min}/2}\}.
$
Let $j\in\mathbb{N}$ be even. 
The same exact reasoning as in \emph{Step 3} from the proof of Proposition \ref{PropBoundVarLevels} allows to define a continuous-time martingale $D^s$ such that $D^s = (\rho^s_{h,\Dt} - \overline{\rho}^s_{h,\Dt}, \phi^{s}_{h,\Dt})_h$ for all $\mathcal{S}_\tau \ni s\leq t$.
Using the continuous It\^o formula, the bound $\lambda_{\max}(I_d-b_0\Dt\Delta_h)^{-1}) \leq 1$, the definition of $\Ts$, and the notation $s_{\leftarrow} := \max\{m \in \Tgrid[\Dt]: m < s \}$,
we obtain
\begin{align*}
\mean{(D^{t\wedge \Ts})^j} & \leq \mean{(D^0)^j} + \frac{j^2}{N}\mean{\int_{0}^{t \wedge \Ts}(D^s)^{j-2} ([\rho_{h,\Dt}^{s_{\leftarrow}}]^+, |\nabla_h \phi^{s_{\leftarrow}}_{h,\Dt}|^2)_h \m s} \nonumber \\
& \quad \lesssim \mean{\|\rho^0_{h,\Dt} - \ov{\rho}^0_{h,\Dt}\|^j_{\infty}}\|\phi^0_{h,\Dt}\|_{L^1}^j \nonumber \\
& \quad \quad  + \frac{j^2}{N}\mean{\sup_{s \leq t\wedge \Ts}(D^{s})^{j-2}}\rho_{\max}(B+1)\!\!\!\sum_{s \in \Tgrid[\Dt], s\leq t}\!\!\!\!\!\|\nabla_h \phi^{s}_{h,\Dt}\|^2_h \Dt. 
\end{align*}
Lemma \ref{Discrete_Energy_Est} with the fact that $\|\phi_{h,\Dt}^{t}\|^2 \propto h^{-d}$, and Doob' martingale inequality imply
\begin{align*}
\mean{\sup_{s \leq t\wedge \Ts}(D^{s})^{j}} & \lesssim \mean{(D^{t\wedge \Ts})^j}  \nonumber \\
&  \leq (N^{-1}h^{-d})^\frac{j}{2} + N^{-1}j^2\mean{\sup_{s \leq t\wedge \Ts}(D^{s})^{j-2}}\rho_{\max}(B+1)(\cfl^\ths \vee 1) h^{-d}. 
\end{align*}
Using H\"older and Young inequalities with conjugate exponents $j/2, j/(j-2)$ and absorbing the term $\mean{\sup_{s \leq t\wedge \Ts}(D^{s})^{j}}$, we arrive at
\begin{align}\label{ItoPositivity_3}
\mean{\sup_{s \leq t\wedge \Ts}(D^{s})^{j}}
 \lesssim \left(Cj^2\frac{(B+1)\rho_{\max}(\cfl^{\ths} \vee 1)}{Nh^d}\right)^{j/2} 
\end{align}
%\begin{align}\label{ItoPositivity_3}
%\mean{\left(\rho^{t \wedge \Ts}_{h,\Dt} - \ov{\rho}_{h,\Dt}^{t \wedge \Ts}, \phi_{h,\Dt}^{t \wedge \Ts}\right)_h^j}
% \lesssim \left(Cj^2\frac{(B+1)\rho_{\max}(\cfl^{\ths} \vee 1)}{Nh^d}\right)^{j/2}.
%\end{align}
Using Chebyshev's inequality and \eqref{ItoPositivity_3}, we get
\begin{align}\label{Chebishev}
\mathbb{P}\left[ \sup_{s\leq t \wedge \Ts} |D^s| \geq \frac{B\rho_{\min}}{4} \right] & \lesssim \frac{\mean{\sup_{s \leq t\wedge \Ts}(D^{s})^{j}}}{(B\rho_{\min})^j} \nonumber\\
& \lesssim \left(\sqrt{C}\frac{\sqrt{B+1}}{B}\frac{\sqrt{\rho_{\max}}}{\rho_{\min}}\left\{\frac{\cfl^{\ths} \vee 1}{Nh^d}\right\}^{1/2}\right)^j \cdot j^j =: \eta^j \cdot j^j,
\end{align}
where we have set 
\begin{align}\label{define_alpha_beta}
\eta := \alpha^{-1} \cdot \frac{\sqrt{B+1}}{B}, \qquad
\alpha =: \left(\sqrt{C}\frac{\sqrt{\rho_{\max}}}{\rho_{\min}}\left\{\frac{\cfl^{\ths} \vee 1}{Nh^d}\right\}^{1/2}\right)^{-1}.
\end{align}
If $\eta \lesssim 1$, we can perform an optimisation argument over $j\geq 2$ for the right-hand-side of \eqref{Chebishev}. Thus the optimal $j$ is chosen as either $\lceil e^{-1}\eta^{-1}\rceil$ or $\lfloor e^{-1}\eta^{-1} \rfloor$. 
In either case, this optimisation procedure allows to carry on in \eqref{Chebishev} and obtain (for a different constant $C$ rescaled by $e^{-1}$)
\begin{align}\label{Optimised}
& \mathbb{P}\left[ \sup_{s\leq t \wedge \Ts} |D^s| \leq \frac{B\rho_{\min}}{4} \right] \lesssim \exp\left(-\frac{C\rho_{\min}}{\rho_{\max}^{1/2}}\!\cdot \left\{\frac{Nh^{d}}{\cfl^\ths \vee 1}\right\}^{1/2}\!\!\!\cdot\frac{B}{\sqrt{B+1}}\right).
\end{align}
Equation \eqref{Optimised} holds also in the case $\eta \gtrsim 1$, thanks to the trivial bounds
\begin{align*}
\mathbb{P}\left[ \sup_{s\leq t \wedge \Ts} |D^s| \geq \frac{B\rho_{\min}}{4} \right] \leq 1 
\lesssim 
e^{-1} \leq e^{-\eta^{-1}}.
\end{align*}
In particular, \eqref{Optimised} holds for all $B\geq 0$. Using \eqref{Optimised}, and the fact that the test function $\delta_h := \mathbf{1}_{y = x}h^{-d}$ (being a discrete Dirac delta at point $x$) recovers the value $\rho^t_{h,\Dt} - \overline{\rho}^t_{h,\Dt}$, we obtain 
%\begin{align*}
%& \mathbb{P}\left[ \sup_{s\leq t \wedge \Ts} |D^s| \geq \frac{B\rho_{\min}}{4} \right] \leq \exp\left(-\frac{C\rho_{\min}}{\rho_{\max}^{1/2}}\!\cdot \left\{\frac{Nh^{d}}{\cfl^\ths \vee 1}\right\}^{1/2}\!\!\!\cdot\frac{B}{\sqrt{B+1}}\right).
%\end{align*}
%Using the definition of $D^s$, we then obtain
\begin{align}\label{ItoPositivity_4}
& \mathbb{P}\left[ t\leq \Ts, \left|\rho^{t}_{h,\Dt} - \ov{\rho}_{h,\Dt}^{t}\right|(x) \geq \frac{B\rho_{\min}}{4} \right] \nonumber\\
& \quad \lesssim \exp\left(-\frac{C\rho_{\min}}{\rho_{\max}^{1/2}}\cdot \left\{\frac{Nh^{d}}{\cfl^\ths \vee 1}\right\}^{1/2}\cdot\frac{B}{\sqrt{B+1}}\right),\qquad \forall B\geq 0.
\end{align}
Now \eqref{eqn_ErrToMFL} follows from applying \eqref{ItoPositivity_4} over the union of all $ (t,x)\in\mathcal{S}_\tau \times G_{h,d}$.

We now move to proving \eqref{MomentBound_infty}. Assume that $\alpha \geq 1$, where $\alpha$ is defined in \eqref{define_alpha_beta}. Then the inequality \eqref{MomentBound_infty} is deduced from \eqref{eqn_ErrToMFL} using the equality $\mean{X} = \int_0^\infty \mathbb{P}(X\geq x)\m x$ (which is valid for every non-negative real-valued random variable X), a simple change of variables, and Lemma \ref{lma_ExpIntBd} with the aforementioned $\alpha$. Namely, 
\begin{align*}
\mean{\sup_{t\in \Tgrid[\Dt]} \|\rho^{t}_{h,\tau}-\ov{\rho}_{h,\tau}^t\|^j_{\infty}} & = \int_0^\infty \mathbb{P}\left(\sup_{t\in \Tgrid[\Dt]} \|\rho^{t}_{h,\tau}-\ov{\rho}_{h,\tau}^t\|^j_{\infty}\geq x\right)\m x \\
& = \int_0^\infty \mathbb{P}\left(\sup_{t\in \Tgrid[\Dt]} \|\rho^{t}_{h,\tau}-\ov{\rho}_{h,\tau}^t\|_{\infty}\geq B^{1/j}\rho_{\min}\right)\rho_{\min}^{j}\m B \\
& \lesssim \left\{C^{-1/2}(\rho_{\min}/\rho_{\max}^{1/2}) [ Nh^{d}/(\cfl^\ths \vee 1)]^{1/2}\right\}^{-j}\rho_{\min}^{j}.
\end{align*}
and \eqref{MomentBound_infty} is proven.
Adapting the proof of \eqref{MomentBound_infty} in the case $\alpha < 1$ so as to get a bound which is independent of $\rho_{\min}$ can be done by replicating the analysis leading up to \eqref{ItoPositivity_4} by replacing $\rho_{\min}$ with $\rho_{\max}$ in the definition of the stopping time $\Ts$ and all relevant thresholds (we omit the details).
Finally, \eqref{MomentBound} is easily proved along the lines of \cite[Lemma 16]{cornalba2021dean}.
\end{proof}

{\bfseries Data availability statement}.
The datasets generated and analyzed during the current study are available from the corresponding author on reasonable request.

{\bfseries Conflict of interest}.
The authors have no relevant financial or non-financial interest to disclose.

{\bfseries Acknowledgments}. The authors wish to thank the anonymous referees for their very careful reading of the manuscript and for their valuable suggestions and comments, which led to a substantially improved presentation of the paper.

The authors also wish to thank Quinn Winters for useful preliminary discussions on the subject of this paper.

\bibliographystyle{abbrvnat}
\bibliography{fluctuations_JF}

\end{document}